\documentclass{article}
\usepackage[utf8]{inputenc}
\usepackage{ulem}
\usepackage[title]{appendix}
\usepackage{enumitem}
\title{Smooth Anosov Katok Diffeomorphisms With  Generic Measure}
\author{Divya Khurana}
\date{}
\usepackage{divya}

\newtheorem{maintheorem}{Theorem}

\begin{document}

\maketitle
\begin{abstract}
%{\rr This is an experimental version I am trying to work on which in my opinion could vastly generalize the existing paper and potentially add to the value.} 
We construct a plethora of Anosov-Katok diffeomorphisms with non-ergodic generic measures and various other mixing  and topological properties.  We also construct an explicit collection of the set containing the generic points of the system with interesting values of its Hausdorff dimension.

\end{abstract}
%\keywords{Generic points, , three, four}
%\import{./}{bibliography.tex}
%%%Section-1
%{\rr I am writing down some general remarks here that are to be taken care of before final submission. These issues appear several times in the paper:
%\begin{enumerate} \item Stick to one terminology: Approximation by conjugation, Approximation by conjugation or AbC. Currently all the mixed up. \item Grammar, spelling and capitalization errors (use capital letters only when starting a sentence after a full stop or refering to the name of a person, not after a comma). Please remove those. You may seek help from gramarly or some profession software if you find hard to keep track. \item address all red and blue comments.\item remove this portion later.\end{enumerate}}
\section{Introduction}

In 1970 Anosov and Katok introduced the so called \textit{approximation by conjugation} method (also known as the \textit{Anosov-Katok} or the \textit{AbC} method) to construct examples of transformations satisfying a pre-specified set of topological and/or measure theoretic properties. In the realm of smooth (or in some cases, real-analytic or even symplectic) zero entropy diffeomorphisms, this technique till date remains one of the rare methods that one can use to explore the possibility of the existence of diffeomorphisms satisfying such a set of properties. Such transformations or diffeomorphisms often are important in their own right. However, more interestingly, in recent years, there have been situations where they have been able to exhibit connections, such as that of rotation number at the boundary with the dynamical behaviour of a diffeomorphism \cite{FS}. This method has gained further momentum with the body of work produced by Foreman and Weiss \cite{FW04},\cite{FW1}  establishing anti-classification theorems for smooth diffeomorphisms.

In this article, we wish to explore the construction of various types of Anosov-Katok diffeomorphisms which supports non-ergodic generic measures. For a probability preserving dynamical system $(M,\mathcal{B},\mu,T)$, we define the set of \textit{$\mu$-generic points} 
\begin{align*}
    L_\mu=\{x\in M:\lim_{n\to\infty}\frac{1}{n}\sum_{i=0}^nf(T^ix)=\int_Xfd\mu\; \forall\; f\in C_c(M)\}
\end{align*}
where $C_c(M)$ is the set of all compactly supported real valued continuous functions. The measure $\mu$ is called a \textit{generic measure} if $L_\mu\neq\emptyset$. The celebrated Birkhoff ergodic theorem asserts that $\mu(L_\mu)=1$ for an $\mu$-ergodic transformation.

There has been a considerable amount of interest regarding the existence of generic measures, particularly in the realm of interval exchange transformations. Chaika and Masur \cite{Cha} showed that there exists a minimal non-uniquely ergodic interval exchange transformation on $6$ intervals with $2$ ergodic measures, which also has a non-ergodic generic measure. Later, Cyr and Kra \cite{Cy19} found a criterion for establishing upper bounds on the number of distinct non atomic generic measures for subshifts based on complexity, and as a consequence, they showed that for $k>2$, a minimal exchange of $k$ intervals has at most $k-2$ generic measures. On the other hand, Gelfert and Kwietniak \cite{Gel18} gave an example of a topologically mixing subshift that can have exactly two ergodic measures, none of whose convex combination is generic. 

%{\rr To add: Light intro to the Anosov-Katok spaces, we need even more explicit conjugation maps and precise norm estimates}
%The first example of weakly 
Anosov and Katok, in \cite{AK}, constructed  examples of smooth measure preserving diffeomorphism, which is weakly mixing in the space $\A(M) = \overline{\{h\circ S_t\circ h^{-1}: t\in \T^1, h \in \text{Diff}^{\infty}(M,\mu)\}}^{C^{\infty}}$, on any manifold admitting a non-trivial $\T^1$ action. Later Fayad and Saprykina produced the smooth weakly mixing diffeomorphism in the restricted space $\A_{\alpha}(M) = \overline{\{h\circ S_{\alpha}\circ h^{-1}: h \in \text{Diff}^{\infty}(M,\mu)\}}^{C^{\infty}}$ for any Liouville number $\alpha$, i.e. for each $n$, there exist integers $p>0$ and $q>1$ such that $0<|\alpha-\frac{p}{q}|<\frac{1}{q^n}$. Both the above constructions are built using the approximation by conjugation method: The diffeomorphism is obtained as the limits of sequences $T_n= H_nS_{\alpha_{n+1}}H_n^{-1}$ where $\alpha_{n+1}\in \mathbb{Q}$ and $H_n=h_1\ldots h_n$ where $h_n$ is a measure preserving diffeomorphism satisfying $S_{{\alpha}_n}\circ h_n= h_n\circ S_{{\alpha}_n}$. For the diffeomorphism $T_n$  to converge in the space $\A_{\alpha}(M)$, for any $\alpha$, it needs construction of more explicit conjugation maps $h_n$ and very precise norm estimates and is generally difficult when compared to convergence in the space $\A(M)$. 

In general, a uniquely ergodic measure preserving transformation on a compact metric space is minimal on the support of the measure, but the converse is not true. Markov produced the first counterexample. Further,  Windsor, in (\cite{Win}), constructed a minimal measure preserving diffeomorphism in $\A_{\alpha}(M)$ with the finite number of ergodic measures. 
Afterwards, Banerjee and Kunde(\cite{BaKu}) produced a similar result for the real analytic category on $\T^2$.

It is well known that the Anosov-Katok constructions allow great flexibility, and we present several results in this article that explore the existence of non-ergodic generic measures in this setup. We also note that our constructions will be smooth and, in some cases, even real-analytic. Hereby we extend the above results to produce more compelling examples with different measure-theoretical and topological dynamical properties.
\begin{maintheorem} 
For any natural number $r$, and any Liouvillian number $\a$, there exists a minimal $T\in\mathcal{A}_\a(\T^2)$ such that the Lebesgue measure is a generic measure for $T$, and there exists $r$ absolutely continuous w.r.t. to Lebesgue measures $\mu_1,\mu_2,\ldots,\mu_r$ such that $T$ is weakly mixing w.r.t. each of these measures.
\end{maintheorem}
In fact the approximation by conjugation method on $\T^2$ offers enough flexibility to repeat the construction using block-slide type of maps (\cite{BaKu}, Theorem E) and get the result in the analytic set-up.

\begin{maintheorem} 
For any natural number $r$, there exist a minimal real-analytic $T\in\text{Diff }^\omega(\T^2,\mu)$ constructed by the approximation by conjugation method, such that the Lebesgue measure is a generic measure for $T$, and there exists $r$ absolutely continuous w.r.t. to Lebesgue measures $\mu_1,\mu_2,\ldots,\mu_r$ such that $T$ is weakly mixing w.r.t. each of these measures.
\end{maintheorem}
One of this paper's objectives is to examine the generic points and try to estimate their size. Here, instead of measuring the set of generic points by the Lebesgue measure, we produce more interesting values of their Hausdorff dimension.
%{\textbf{\it{Generic birkhoff spectra}}}
%One of the objectives of this paper is to the existence of the generic points and try to estimate their size. Here, instead of measuring the set of generic points by the Lebesgue measure, we produce more interesting values of their Hausdorff dimension. 
%\begin{maintheorem}There exist a smooth diffeomorphism $T\in\text{Diff }^\infty(\T^2,\mu)$ constructed by the approximation of conjugation method, such that the set $A$ containing all the generic points of $T$ has  $$\text{dim}_H(A)= 1+\log_{3}2$$ where the set is defined as $A= \mathbb{T}^1\times C,$ where $C$ is $1/3$ Cantor set defined on unit interval and $\mu(A)=0.$ {\rr restructure theorem to put the set first and the hausdorff dimension later}
\begin{maintheorem}
There exist a smooth diffeomorphism $T\in\text{Diff }^\infty({\T}^2,\mu)$ constructed by the approximation of conjugation method, 
such that the set $\mathrm{B}$ containing all the generic points of $T$ has  $$\log_{3}2\leq \text{dim}_H(\mathrm{B})\leq 1+\log_{3}2,$$   and $\mu(\mathrm{B})=0.$
\end{maintheorem}
We can generalize the above result by choosing the generalized Cantor set (p-series Cantor set \cite{CC}) instead of the Cantor set in the above setup and construct the generic sets of different Hausdorff dimensions as
\begin{maintheorem}
For any $1<\alpha<2,$ there exist a smooth diffeomorphism $T\in\text{Diff }^\infty(\mathbb{T}^2,\mu)$ constructed by the approximation of conjugation method, such that the set $\mathrm{B}_{\alpha}$ containing all the generic points of $T$ has  $$\alpha-1 \leq \text{dim}_H(\mathrm{B}_{\alpha})\leq \alpha,$$ 
and $\mu(\mathrm{B}_{\alpha})=0.$
%where $C_{\lambda}$ is a generalised cantor set associated to the sequence  $\lambda=\{\lambda_k\}_{k\in \mathbb{N}},$ where $\lambda_k= \frac{1}{c_0}(\frac{1}{k})^{\frac{1}{\alpha-1}},$ the constant $c_0= \sum_{k\in \mathbb{N}}\lambda_k,$ and $\text{dim}_H(C_{\lambda})=\alpha-1$  
\end{maintheorem}

%In $\cite{ZBR},$ Theorem- 2.3.1, the author presented the variational formula result in the  space of the full-shift on an alphabet of twosymbols $(\Omega,\sigma)$. But the latter fact is not valid in general and follows from the next Corollary.
 In $\cite{ZBR},$ Theorem- 2.3.1, the author presented a variational type formula for the full-shift on an alphabet of two
symbols $(\Omega,\sigma)$. But in our set-up, it appears that this theorem does not hold. For example, if $f:\T^2\to\R$  is any continuous function and $\a=\int f d\mu$ with $<1\alpha<2$ and $\mu$ being the usual Lebesgue measure, then the Hausdorff dimension of $E_f(\a)$ is greater than $zero$ (see theorem D) where $E_{f}(\alpha)= \{x\in \mathbb{T}^2 \ : \lim\limits_{N\longrightarrow\infty} \frac{1}{N}\sum\limits_{n=1}^{N}f(T^nx)= \alpha \}$. Whereas, according to the theorem A in $\cite{FFW}$ and theorem 2.3.1 in $\cite{ZBR}$, this number should be zero as topological entropy and all measure theoretic entropy, in our case, is always zero. 

%{\bb \begin{corollary}({\it{Variational formula doesnot hold in general !}})Suppose $f:\T^2\longrightarrow \mathbb{R}$ is a continuous function and $L_{f}= \{\alpha\in \R^2: \alpha= \lim\limits_{N\longrightarrow\infty}\frac{1}{N}\sum\limits_{n=1}^{N}f(T^nx) \ \text{for some}\  x\in \T^2\}.$ Denote the level set of map $T$ as $E_{f}(\alpha)= \{x\in \mathbb{T}^2 \ : \lim\limits_{N\longrightarrow\infty} \frac{1}{N}\sum\limits_{n=1}^{N}f(T^nx)= \alpha \}$. %and  {\rr Something is still wrong, how can the limit you describe be in $\T^2$?} 
 %There exist an $\alpha_0\in L_{f}$ such that$$S_f({\alpha_0})= \text{dim}_{H}(E_{f}(\alpha_0))\neq \max_{\mu\in \mathcal{M}_f}\frac{h_{\mu}}{\log2},$$ where $h_{\mu}$ is the metric entropy of measure $\mu.$ 
%\end{corollary}}
%Hint: For $\alpha = \int_{\T^2}{f}d\mu$ with $1<\|f\|_0<2$, there exist $T\in \text{Diff}^{\infty}(\T^2,\mu)$ with a set $E_{f}({\alpha})= \mathrm{G}_{\lambda}$, containing all the generic points of $T$, has non-zero Hausdorff dimension. Since the map $T$ is obtained by approximation by conjugation method, it has metric entropy equal to zero (refer to {\cite{FK}}).

%Instead of measuring the set of generic point of a map, one can measure the set of non-generic points and can have interesting 
For an ergodic transformation, the set of non-generic points has measure zero but can have
more exciting values of its Hausdorff dimension. Precisely, one can obtain the analogue result of theorem D for the set of non-generic points for the case of ergodic measure with the appropriate choice of combinatorics.

\begin{maintheorem}
%For any $1<\alpha< 2,$ there exists a smooth ergodic diffeomorphism $T\in \text{Diff}^{\infty}(\mathbb{T}^2, \mu)$ constructed by the approximation of the conjugation method, such that $A_{\alpha}$ consists of all the non-generic points of T has $$\text{dim}_{H}(A_{\alpha})= \alpha.$$ {\rr notation $A_\a$ is confusing... seems like an anosob katok space.}
For any $1<\alpha<2,$ there exist a smooth ergodic diffeomorphism $T\in\text{Diff }^\infty(\mathbb{T}^2,\mu)$ constructed by the approximation of conjugation method, such that the set $\mathrm{B}_{\alpha}$ containing all the non-generic points of $T$ has  $$\alpha-1 \leq \text{dim}_H(\mathrm{B}_{\alpha})\leq \alpha,$$
and $\mu(\mathrm{B}_{\alpha})=0.$
%where $C_{\lambda}$ is a generalised cantor set associated to the sequence  $\lambda=\{\lambda_k\}_{k\in \mathbb{N}},$ where $\lambda_k= \frac{1}{c_0}(\frac{1}{k})^{\frac{1}{\alpha-1}},$ the constant $c_0= \sum_{k\in \mathbb{N}}\lambda_k,$ and $\text{dim}_H(C_{\lambda})=\alpha-1$ 
%For any $1<\alpha<2,$ and a set $\mathrm{G_{\lambda}}= \mathbb{T}^1\times C_{\lambda},$ where $C_{\lambda}$ is a generalised Cantor set with $\lambda=\{(\frac{1}{k})^{\frac{1}{\alpha-1}}\}$ defined on unit interval and $\mu(\mathrm{G}_\lambda)=0.$ There exist a smooth ergodic diffeomorphism $T\in\text{Diff }^\infty(\T^2,\mu)$ constructed by the approximation of conjugation method, such that the set $\mathrm{G}_{\lambda}$ contains all the non-generic points of $T$
 
\end{maintheorem}
\remark The diffeomorphism produced in the above Theorem-C, D, and E could be made minimal by following the same construction as in theorem A.
\section{Preliminaries}
This section explains some basic definitions and standard techniques that we use throughout the paper.

\subsection{Basics of ergodic theory}
Consider $(X,d)$ be a $\sigma$-compact metric space, $\mathcal{B}$ is a $\sigma$ algebra, $\mu$ is a measure and $T:X\longrightarrow X$ is a measure preserving transformation(\it mpt) i.e. $\mu(T^{-1}(A))= \mu(A) \ \forall A\in \mathcal{B}$. 
\begin{definition}
A {\it mpt} $(X,\mathcal{B},\mu,T)$ is called {\it ergodic} if every invariant set $E\in \mathcal{B}$ satisfies $\mu(E) = 0 \ or \ \mu(X \backslash E) = 0$. We say $\mu$ is ergodic measure.
\end{definition}
\definition \label{def:1a} A point $x\in X$ is a {\it generic point} for $\mu$ if for every continuous compactly supported $\phi:X \longrightarrow \mathbb{R}$, we have $\frac{1}{N}\sum\limits_{i=0}^{N-1} \phi(T^ix)\longrightarrow \int \phi d\mu.$\\
A measure is called  {\it generic measure} if it has a generic point.
It follows from the Birkhoff ergodic theorem that if the system is ergodic, then $\mu$ almost-every point is generic. 
\definition Let $T:X\longrightarrow X$ be a continuous map where $X$ is topological space. The map $T$ is said to be minimal if for every $x\in X,$ the orbit $\{T^i(x)\}_{i\in \mathbb{N}}$ is dense in $X.$
Equivalently, in the case of a metric space, the map $T$ is minimal if for every $x\in X$,  $\delta>0$ and every $\delta$-ball $B_{\delta}$ there exist $i \in \mathbb{N}$ such that $T^i(x)\in B_{\delta}.$ 
 \definition 
 A measure preserving
diffeomorphism $T : X \longrightarrow X$ is said to be weakly mixing on the space $(X,\mathcal{B},\mu, T)$ if there exists a sequence $\{m_n\}\in \mathbb{N}$
such that for any pair $A,B \in \mathcal{B}:$
 $$ {\left\lvert \mu(B\cap f^{-m_n}(A)) - \mu(B)\mu(A) \right\rvert} \longrightarrow 0.$$
\subsection{The middle third Cantor set}\label{sec:2.2a}
Consider a middle third Cantor set $C\subset[0,1],$  obtained  by removing the open middle third interval and then repeating the same process with each remaining interval.
After completing the $n$ stage of removing middle intervals from $[0,1]$, we have $2^n$ number of closed intervals enumerated as $I_l^n, \ l=0,1,...,2^n-1$ and have $2^{n-1}$ number of removed open interval denoted as $J_l^n, \ l=0,1,...,2^{n-1}-1.$ Precisely, the interval $I_l^n$ is of the 
form  $\left[\frac{3k}{3^n},\frac{3k+1}{3^n}\right]$ or $\left[\frac{3k+2}{3^n},\frac{3k+3}{3^n}\right]$, and interval $J_l^n$ of the form $\left(\frac{3k+1}{3^n},\frac{3k+2}{3^n}\right)$, for $k = 0,1,\ldots,3^{n-1}-1$. The explicit closed form of the Cantor set is defined as 
\begin{align}
    C=\bigcap\limits_{n\geq1}\bigcup_{l=0}^{2^{n}-1} I_l^n = \left[0,1\right]\backslash \bigcup_{n=1}^{\infty} \bigcup_{l=0}^{2^{n-1}-1} J_l^n
\end{align}

\subsection{The Cantor set associated with a sequence}\label{sec:2.3a}
For any sequence $\lambda=\{\lambda_k\}_{k\in \mathbb{N}}$ such that $\sum \lambda_k= K,$ there exists a Cantor set $C_{\lambda}$ associated with it, defined on the interval $I_{0,\lambda}=[0,K]$ and also known as generalised Cantor Set. It is constructed in a similar way to the middle third Cantor set and has the same topological and measure properties. Precisely, it is a compact, perfect, totally disconnected subset of the real line and has measure zero. \\
The set $C_{\lambda}$ is obtained by the removal of open intervals whose lengths are the terms of the sequence $\lambda$. In the first step, an open interval $J_{0,\lambda}^1$ of length $\lambda_1$ is removed from $I_{0,\lambda}$, obtaining two closed intervals $I_{0,\lambda}^1, I_{1,\lambda}^1$. In the second step, we remove an open interval of length $\lambda_2$ and ${\lambda_3}$ from $I_{0,\lambda}^1$ and $I_{1,\lambda}^1$, respectively. After $k$ complete steps, we have $2^{k}$ number of closed intervals denoted as $\{I_{l,\lambda}^k\}_{l=0}^{2^k-1}$ and $2^{k-1}$ number of removed open intervals denoted as $\{J_{l,\lambda}^k\}_{l=0}^{2^{k-1}-1}$ of length equal to the 
previously used terms of the sequence. And continue in this way,
removing an open interval $J_{l,\lambda}^{k+1}$ of length $\lambda_{2^k+l}$ from interval $I_{l,\lambda}^k$ we have $I_{2l,\lambda}^{k+1}$ and $I_{2l+1,\lambda}^{k+1}$. 
Since $\sum_{k} \lambda_k= K,$ the location of each interval $J_{l,\lambda}^k$ to be removed is determined uniquely, and the Cantor set $C_{\lambda}$ is well defined as
\begin{align}
    C=\bigcap\limits_{n\geq1}\bigcup_{l=0}^{2^{n}-1} I_{l,\lambda}^{n} = \left[0,K\right]\backslash \bigcup_{n=1}^{\infty} \bigcup_{l=0}^{2^{n-1}-1} J_{l,\lambda}^{n}
\end{align}
\remark  Since the length of the
interval $I_{0,\lambda}$ equals the sum of the lengths of all the intervals removed in the
construction, and there is a unique way of doing this construction.
\remark Clearly, by normalization, we can define $C_{\lambda}$ on $I_0=[0,1]$ for the sequence $\lambda$. In our case, we choose Cantor sets on $[0,1],$  associated with the sequence $\lambda=\{\lambda_k\}_{k\in \mathbb{N}},$ where $\lambda_k= \frac{1}{c_0}(\frac{1}{k})^p$ such that $c_0= \sum_{k\in \mathbb{N}}\lambda_k$ (The constant $c_0$ is finite only for case $p>1$), and its Hausdorff dimension is  described in more detail in \cite{CC}, 
\begin{align}
\text{dim}_H(C_{\lambda})= \frac{1}{p} \label{eq:6.1d}
\end{align}

\remark If X and Y are metric spaces, then the Hausdorff dimension of their product satisfies 
\begin{align}
\text{dim}_H(X)+ \text{dim}_H(Y)\leq \text{dim}_H(X\times Y)\leq \text{dim}_H(X) + \text{dim}_B(Y)    
\end{align} where $\text{dim}_B$ is the upper box counting dimension (see \cite{Ma54}). In particular, if $Y$ has equal Hausdorff and upper box-counting dimension (which holds if $Y$ is a compact interval), then 
 \begin{align}
     \text{dim}_{H}(X\times Y) = \text{dim}_{H}(X)+ \text{dim}_{H}(Y) \label{def:1b}
\end{align}

\subsection{Smooth and Real-analytic diffeomorphisms}
For the description of standard topology on the space of diffeomorphism on $M=\mathbb{T}^2$ and,  explicitly, convergence in the space of smooth diffeomorphism and real-analytic diffeomorphism on the torus, one can ref to \cite{FS}.
\subsection{Approximation by conjugation method}
Here, we outline a scheme of constructing a smooth area preserving diffeomorphism with the specific ergodic property via the Approximation by conjugation method explained in \cite{AK}. Let's denote $S_t$, a measure preserving circle action $\mathbb{T}^1$ on the torus $\mathbb{T}^2= \mathbb{R}/\mathbb{Z}\times \mathbb{R}/\mathbb{Z}$ defined as a translation $t$ in the first coordinate : $S_{t}(x_1,x_2)= (x_1+t,x_2).$
The required map $T$ is constructed as the limit of a sequence of periodic measure preserving diffeomorphism $T_n$ in the smooth topology. The sequence of $T_n$ is defined iteratively as
     \begin{align}\label{eq:1d}
     T_n = H_n\circ S_{\alpha_{n+1}}\circ H_n^{-1}. 
     \end{align}
     where $\alpha_{n+1}= \frac{p_{n+1}}{q_{n+1}}\in \mathbb{Q}/\mathbb{Z}$ and $H_n\in \text{Diff}^{\infty}(\mathbb{T}^2)$.
   The diffeomorphism $H_n$ 
    is constructed successively as $H_n = h_1\circ \ldots \circ h_n,$ where $h_n$ is an area preserving diffeomorphism of $\mathbb{T}^2$ that satisfies 
    \begin{align}\label{eq:2d}
    h_n\circ S_{\alpha_{n}}= S_{\alpha_{n}}\circ h_n.
    \end{align}
    The rationals $\alpha_{n+1}=\frac{p_{n+1}}{q_{n+1}}$ are defined iteratively as $p_{n+1} = k_{n}l_{n}q_{n}p_{n}+ 1$ and $q_{n+1} = k_{n}l_{n}q_{n}^2$ where  $\{k_{n}\},\{l_{n}\}$ is the sequence of natural numbers chosen such that $\alpha_{n+1}$ is close enough to $\alpha_{n}$ to ensure the closeness between $T_n$ and $T_{n-1}$ in the $C^{\infty}$ topology. Given $\alpha_{n+1}, H_n$, at the $n+1$ stage of this iterative process, we construct $h_{n+1}$ such that $T_{n+1}$ satisfy a finite version of the specific property we eventually need to achieve for the  limiting diffeomorphism. The explicit construction of $h_{n+1}$ has been done in section 3, which serves our purpose.
    Then we construct $\alpha_{n+2}= \alpha_{n+1} + \frac{1}{k_{n+1}l_{n+1}q_{n+1}^2}$ by choosing $k_{n+1}\in \mathbb{N}$ and $l_{n+1}\in \mathbb{N}$ to be large enough such that it satisfies the certain condition
    %\begin{align}\label{eq:2a}  l_n >2^n\|DH_n\|_0
    %\end{align}
    %{\begin{align}\label{eq:2a}
    %l_n > 2^n\|DH_n\|_{n-1}
    %\end{align}}
    and guarantees the convergence of iterative sequence $T_{n+1}$ in the smooth topology.
The limit obtained from this induction sequence is the required smooth diffeomorphism with the specific ergodic and/or topological properties, $T_{n+1} \longrightarrow T \in \text{Diff}^{\infty}(\mathbb{T}^2,\mu).$

\subsection{Preliminary Lemma}
\begin{lemma}\label{le:2a}
Let $g,h \in \text{Diff}^{\infty}(\mathbb{T}^2)$. For $k \in \mathbb{N},$  the norm estimates of the  composition $g \circ h$ satisfy
\begin{align}
    \vertiii{g\circ h}_k \leq C \vertiii{g}_k^k.\vertiii{h}_k^k,
\end{align}
where $C$ is constant.
\end{lemma}
\remark The above can be deduced using the corollary of the Faa di Bruno formula; similar proof has been done in [\cite{Ku15}, lemma 4.1].
%\begin{lemma}[\cite{FS}, Lemma 5.3]For any $0<\e<1/2$, there is a smooth Lebesgue measure preserving diffeomorphism $\varphi=\varphi(\e)$ of $\R^2$, equal to an identity outside $[-1+\e,1-\e]^2$ and rotating the square $[-1+2\e,1-2\e]^2$ by $\pi/2$ in the clockwise direction.\end{lemma}
\begin{lemma}\label{lem:01}
For any $\e>0$, there is a smooth Lebesgue measure preserving diffeomorphism $\varphi=\varphi(\e)$ of $[0,1]^2$, equal to identity outside $[\e,1-\e]^2$ and rotating the square $[2\e,1-2\e]^2$ by $\pi/2$ in the clockwise direction.
\end{lemma}
The proof directly follows from [\cite{FS}, lemma 5.3].
%\begin{lemma} \label{lem:3a} Suppose that $X,Y$ are smooth compact Riemannian n-manifolds, $f:X\longrightarrow Y$ is a diffeomorphism. Then $f$ is bilipschitz with respect to the Riemannian distance functions. Moreover, $$\text{dim}_H(f(X))= \text{dim}_H(X).$$\end{lemma}
\begin{lemma}\label{lem:3a}
For any diffeomorphism $\phi:\Delta \longrightarrow \mathbb{R}^n$. For any compact set $A\subset\Delta$ : $$\text{dim}_H(\phi(A))= \text{dim}_H(A)$$
\end{lemma}

\section{Construction of the Conjugacies}
We consider the following conjugacies for the Approximation by conjugation method, for any $0 < \sigma < \frac{1}{2}$, on the torus as
\begin{align} %& H_n=h_1\circ\ldots \circ h_n \label{eq:3a}\\ 
    T_n&=H_n\circ S_{\a_{n+1}}\circ H_n^{-1} \ \text{where} \ H_n=h_1\circ\ldots \circ h_n \label{eq:3b}\\
    h_n&= g_n\circ \phi_n \circ P_n \label{eq:3:d}\\
    g_n(x,y)&=(x+\lfloor nq_n^\sigma\rfloor y,y) \label{eq:3c}
\end{align}
where the sequence $\alpha_{n+1} = p_{n+1}/q_{n+1}$%and $0<\sigma<\frac{1}{2} $,
converging to $\alpha$ (a Liouville number), and the diffeomorphisms $\phi_n$ and $P_n$ commute with $S_{\alpha_n}$, are constructed in section \ref{sec:3a} below. 
\subsection{Outline}
In order to prove theorem A, we decompose torus $\T^2$ into three different parts with distinct aims. On the one hand, we divide $\T^2$ into $r$ disjoint sets as $N^t$ where each set naturally supports an absolutely continuous Lebesgue measure $\mu_t$ obtained by the normalized Lebesgue measure $\mu$. \\
While on the other hand, we introduce another two different parts inside $\T^2$ such that other two dynamics property can be achieved explicitly. These parts are chosen to be measure theoretically insignificant such that the measure of these sets goes to zero. 
Then, with appropriate geometrical and combinatorial criterion explained in the next section, gives us the limit diffeomorphism $T$, obtained by $(\ref{eq:3a})$, to be minimal and have $r$ distinct weak mixing measures $\mu_t$ on $\T^2$ and, Lebesgue measure $\mu$ as a generic measure. 
\subsection{Explicit set-up}
This subsequent section introduces a couple of fundamental domains on which our explicit
construction of conjugation maps exhibits different ergodic properties. 
First, define the following subsets of $\mathbb{T}^2$, for $t=0,\ldots,r-1$: 
\begin{align} \label{eq:3.1b}
    N^t=\mathbb{T}^1\times \left[\frac{t}{r},\frac{t+1}{r}\right]
\end{align}
and denote $\mu_t$ be a measure on $N^t$ defined as normalized Lebesgue measure $\mu$ to $N^t,$ i.e.  $\mu_{t}(A)= \frac{\mu(A\cap N^t)}{\mu(N^t)}$ for measurable set $A\in \mathcal{B}(\mathbb{T}^2).$
%This allows us to show inLemma 5.18 that orbits starting in different domains are Hamming apart from each otherwhich will give us a lower bound on the upper measure-theoretical slow entropy. To getan upper bound we provide in Lemma 5.17 some (1 − ε)-cover with (ε, qn+1)-Hammingballs with respect to a given partition.
Considering the following fundamental domain of $N^t$  for $t\in \{0,\ldots,r-1\}$ as
\begin{itemize}
    \item The fundamental domain: $D_{n}^t=\Big[0,\frac{1}{q_n}\Big]\times \Big[\frac{t}{r}, \frac{t+1}{r}\Big)$. 
    
    \item Split the $D_n^t$ into two halves %by $y=1/2q_n$ 
    : $D_{n}^{t,1}=\Big[0,\frac{1}{2q_n}\Big)\times \Big[\frac{t}{r}, \frac{t+1}{r}\Big)$  and  $D_{n}^{t,2}=\Big[\frac{1}{2q_n},\frac{1}{q_n}\Big)\times \Big[\frac{t}{r}, \frac{t+1}{r}\Big).$ 
%So $D_{n}^{t} = D_{n}^{t,1}\sqcup D_{n}^{t,2}$.
 \item $D_{n,j}^t,$ the 
shift of fundamental domain: 
$D_{n,j}^{t}= S_{j/q_n}(D_n^t),$ and so $D_{n,j}^{t,i}= S_{j/q_n}(D_n^{t,i})$
\end{itemize}

\subsubsection{Construction of the conjugacies}\label{eq:sec1}
%\begin{align*}
    %\phi_n=\sum_{i=0}^{r}\chi_{X_i}\phi_{n,i}
%\end{align*}
%For $i=0,1,..,r-1,$ 
The aim is to construct the conjugation map $h
_n$, which allows the limiting diffeomorphism T, defined by (\ref{eq:3b}), to have $r$ weak mixing measures and  have the Lebesgue measure as a generic measure, and be a minimal map.
Here, we proceed with the construction of conjugation map $\phi_n$ in the following three steps and combining all together; we define the smooth diffeomorphism  $\phi_n:\T^2\longrightarrow\T^2$ as  
\begin{align}
 \phi_n= \phi_n^g\circ\phi_n^m\circ\phi_n^w \label{eq:3a}
 \end{align} 
{\bb{Step-1:-}} Define the map $\phi_{n}^{w} : \mathbb{T}^2\longrightarrow\mathbb{T}^2$ to achieve  $r$ weak mixing  measures supported on each $N^t$: 
\begin{equation}
\phi_{n}^w(x) = 
    \begin{cases} 
      \phi_{n,0}(x)\ \  &{\text{if}} \  x\in N^{0}\\
      \phi_{n,1}(x)\ \  &{\text{if}} \  x\in N^{1}\\
      \vdots\ \  &{} \ \vdots\\
      \phi_{n,r-1}(x)\ \  &{\text{if}} \  x \in N^{r-1}\\
      x \ \  &{\text{otherwise,}}  
   \end{cases} 
\end{equation}
where $\phi_{n,t}$ is a smooth diffeomorphism defined on $\T^2$ for  $t=0,1,\ldots, r-1$ as described in the following paragraph. Consider a map $\phi_{n,t}:\Big[0,\frac{1}{q_n}\Big]\times \Big[\frac{t}{r}, \frac{t+1}{r}\Big) \longrightarrow \Big[0,\frac{1}{q_n}\Big]\times \Big[\frac{t}{r}, \frac{t+1}{r}\Big) :$ 
%we start by choosing any $0<\sigma<1$ and define
\begin{equation}\label{eqn:5.4}
\phi_{n,t} = 
    \begin{cases} 
      C_{n,t}^{-1} \circ \varphi_n^{-1}(\e_n^{(1)})\circ C_{n,t} \ \  &{\text{on}} \  D_{n}^{t,1}\\
      Id \ \  &{\text{otherwise}}  
   \end{cases} 
\end{equation}
%\begin{align*}
    %& \phi_{n,i}=C_{n,i}^{-1}\circ \varphi_n(\epsilon_1)\circ C_{n,i}
%\end{align*}
here $C_{n,t}(x,y)=(q_nx,ry-t)$ and $\varphi$ is defined as in lemma \ref{lem:01} with $ \e_n^{(1)} =1/3nr$. In the same way we can extend this map  $\phi_{n,t}$ as $\frac{1}{q_n}$ -equivariantly
on the whole $N^t$, as done in \cite{FS}. \\
%\remark:Here we modify the $\phi_{n,i}$ for our explicit construction to have  modify the $\phi_n$, 
%here we modify the map $\phi_{n,i}$ that builds on the construction of  
%by introducing the shifts on the domains as done in 
% \remark  
%We define domains  $Y_{r,n}=[0,\d_n]\times \T^1$ and  
{\bb{Step} 2:-} Here, we construct a smooth diffeomorphism  $\phi_{n}^g: \T^2\longrightarrow \T^2$ differently to ensure the existence of a generic point. Consider a map $\phi_{n}^g:\Big[0,\frac{1}{q_n}\Big]\times \mathbb{T}^1 \longrightarrow \Big[0,\frac{1}{q_n}\Big]\times \mathbb{T}^1$ defined as
\begin{align*}
    \phi_{n}^g=\tilde{C}_n^{-1}\circ {\varphi}^{-1}(\e_n^{(3)})\circ{{\varphi}}(\e_n^{(2)})\circ\tilde{C}_n
\end{align*}
where $\tilde{C}_{n}(x,y)=(q_nx,y)$ and $\varphi$ is defined in lemma \ref{lem:01} with the choice of  $\e_n^{(2)}= \frac{\e_{n}^{(1)}}{8}$and $\e_n^{(3)} =\frac{\e_n^{(1)}}{2}.$
As in the above step, we extend the $\phi_n^g$ equivariantly on $\T^2.$\\
%($\boldsymbol{check!}$) carefully such that     effect only Add double rotation effect.
Let's denote $B_{n,i}=\left[\frac{i}{q_n}+\frac{2\e_n^{(2)}}{q_n},\frac{i+1}{q_n}-\frac{2\e_n^{(2)}}{q_n}\right]\times [2\e_n^{(2)}, \e_n^{(3)}]$ and $Y_{n,i}= \left[\frac{i+1}{q_n}- \frac{\e_n^{(3)}}{q_n},\frac{i+1}{q_n}-\frac{2\e_n^{(2)}}{q_n}\right] \times [2\e_n^{(2)}, 1-2\e_n^{(2)}]$ for $i=0,\ldots, q_{n}-1.$

% Let's denote $B_{n}=\left[\frac{2\e_n^2}{q_n},\frac{1-2\e_n^2}{q_n}\right]\times [2\e_n^2, \e_n^3]$ and $Y_{i,n}= \left[\frac{1-\e_n^3}{q_n},\frac{1-2\e_n^2}{q_n}\right] \times [2\e_n^2, 1-2\e_n^2]$ for $i=0,\ldots, q_{n}-1.$
\remark This scheme is so-called as ``double rotation effect", as $\varphi^{-1}(\e_n^{(3)})\circ\varphi(\e_n^{(2)})$ first rotate the whole square with the error $\e_n^{(2)}$, i.e. rotate inside the square $[2\e_n^{(2)},1-2\e_n^{(2)}]^2$, by $\frac{\pi}{2}$ in the clockwise direction and act as an identity outside the square $[\e_n^{(2)},1-\e_n^{(2)}]^2$ (see lemma \ref{lem:01}). Similarly, we rotate the whole square with the error $\e_n^{(3)}$, i.e. $[2\e_n^{(3)},1-2\e_n^{(3)}]^2$, in the anticlockwise direction. Note that with the specific choice of $\e_n^{(2)}$ and $\e_n^{(3)}$, the map $\phi_n^g$ satisfying the following properties:
\begin{enumerate}
    \item $\phi_n^g$ rotates the region $B_{n,i}$ by $\pi/2$ and then transforms $B_{n,i}$ inside $Y_{n,i}$, i.e. $\phi_n^g(B_{n,i})= Y_{n,i}.$ 
    \item $\phi_n^g$ acts as an identity on the region $\Sigma_{1}\cup\Sigma_2,$ where
    \begin{itemize}
        \item $\Sigma_1=\bigcup_{i=0}^{q_n-1}\left(\left[\frac{i}{q_n},\frac{i}{q_n}+\frac{\e_n^{(2)}}{q_n}\right]\cup\left[\frac{i+1}{q_n}-\frac{\e_n^{(2)}}{q_n},\frac{i+1}{q_n}\right]\right)\times \left([0,\e_n^{(2)}]\cup[1-\e_n^{(2)},1] \right)$
         \item $\Sigma_2=\bigcup_{i=0}^{q_n-1}\left[\frac{i}{q_n}+\frac{2\e_n^{(3)}}{q_n},\frac{i+1}{q_n}-\frac{2\e_n^{(3)}}{q_n}\right]\times [2\e_n^{(3)}, 1-2\e_n^{(3)}]$
         \end{itemize}
  %  \item $\phi_n^g$ acts as an identity on the region $\Xi_{1}\cup\Xi_2,$ where
    %\begin{itemize} \item $\Xi_1=\bigcup_{i=0}^{q_n-1}\left(\left[\frac{i}{q_n},\frac{i}{q_n}+\frac{\e_n^2}{q_n}\right]\cup\left[\frac{i+1}{q_n}-\frac{2\e_n^2}{q_n},\frac{i+1}{q_n}\right]\right)\times \left([0,\e_n^2]\cup[1-\e_n^2,1] \right)$ \item $\Xi_2=\bigcup_{i=0}^{q_n-1}\left[\frac{i}{q_n}+\frac{2\e_n^3}{q_n},\frac{i+1}{q_n}-\frac{2\e_n^3}{q_n}\right]\times [2\e_n^3, 1-2\e_n^3]$ \end{itemize}
\end{enumerate}
\remark\label{re:3.5} The region $\E_n^g \subset \T^2\backslash((\cup_{i=0}^{q_n-1}{B_{n,i}\cup Y_{n,i}})\cup (\Sigma_1\cup \Sigma_2))$, say as Error zone, comes from the smoothing of the map $\phi_n^g$.\\
%{\rr do you need this region elsewhere? Mor like this set contains the turbulence zone, rather than being equal}
%To exclude the regions coming from smoothing 
%The map $\phi_n^g$ does not have any control on the turbulence zone 
{\bb{Step 3:-}} In the same spirit, we define 
$R_n = \left[0,\frac{\e_n^{(2)}}{q_n} \right]\times \T^1$
%$R_n = \left[\frac{1}{2q_n}-\frac{1}{2\theta_nq_n},\frac{1}{2q_n}+\frac{1}{2\theta_nq_n} \right]\times \T^1$ where $\theta_n \in \N$ {\rr as discussed, since you know exactly what is $\theta_n$, use that instead of introducing an arbitrary parameter}
and the map $\phi_{n}^m:\Big[0,\frac{1}{q_n}\Big]\times \mathbb{T}^1 \longrightarrow \Big[0,\frac{1}{q_n}\Big]\times \mathbb{T}^1$  differently to achieve minimality as 
\begin{equation}\label{eqn:3.1.2}
\phi_{n}^m = 
    \begin{cases} 
      \hat{C}_{n}^{-1} \circ \varphi(\e_n^{(4)})\circ \hat{C}_{n} \ \  &{\text{on}} \  R_n \\
      Id \ \  &{\text{otherwise}}  
   \end{cases} 
\end{equation}
%\begin{align*}\phi_{n}^m=\tilde{\tilde{C}}_n^{-1}\circ{{\varphi}}_n(\e_4)\tilde{\tilde{C}}_n
%\end{align*}
%where ${C}_{n}(x,y)=(4k_nq_nx-1,2y-1)$ and $\e_4>1/\theta_nq_n$($\boldsymbol{check!}$).\\
where $\hat{C}_{n}(x,y)=(\frac{q_n}{\e_n^{(2)}}x,y)$ and $\e_n^{(4)}= \frac{1}{2^nq_n}.$ We extend the map $\phi_n^m$ equivariantly on $\T^2$ such that it acts as an identity outside the region 
%remark With choice of $\e_4$ and the map $\phi_n^m$ acts on 
$R_{n,i} = \left[\frac{i}{q_n}, \frac{i}{q_n} +\frac{\e_n^{(2)}}{q_n} \right]\times \T^1$ (defined as the shift of domain: $S_{\frac{i}{q_n}}(R_n)= R_{n,i},\  \forall \  i\in \{0,1,\ldots,q_n-1\}$)
\remark With specific chosen  $\e_n^{(4)}$, the map $\phi_n^m$ rotates the  region $\left[\frac{i}{q_n}+ \frac{2\e_n^{(4)}}{q_n} , \frac{i}{q_n} +\frac{\e_n^{(2)}}{q_n}- \frac{2\e_n^{(4)}}{q_n}  \right]\times [2\e_n^{(4)},1-2\e_n^{(4)}]$, inside $R_{n,i}$, by $\pi/2$ and acts as an identity outside the region $R_{n,i}.$ The region 
\begin{align}
\E_n^m = \bigcup_{i=0}^{q_n-1} \bigg(\left[\frac{i}{q_n}+\frac{\e_n^{(4)}}{q_n} ,\frac{i}{q_n}+ \frac{2\e_n^{(4)}}{q_n} \right] &\bigcup \left[\frac{i}{q_n}+ \frac{\e_n^{(2)}}{q_n}- \frac{2\e_n^{(4)}}{q_n},\frac{i}{q_n} +\frac{\e_n^{(2)}}{q_n}- \frac{\e_n^{(4)}}{q_n} \right]\bigg) \times \nonumber\\
&\qquad  \qquad \left([\e_n^{(4)},2\e_n^{(4)}]\cup[1-2\e_n^{(4)},1-\e_n^{(4)}] \right)
\end{align}
the error zone comes from the smoothing of the map $\phi_n^m$ (see Figure 1).
\begin{lemma}
The diffeomorphism $\phi_n$ constructed above satisfy: for all $k\in \mathbb{N}$, $\vertiii{\phi_n}_k\leq c_k(n,k) q_n^{2k^3+k}$ where $c_k(n,k)$ is independent of $q_n.$
\end{lemma}
\begin{proof}
 For any $a\in \mathbb{N}^2$ with $|a|=k$, we have $\|(D_a \phi_n^m)_j\|_0\leq c_m .q_n^{k},$ and similarly, $\|(D_a (\phi_n^m)^{-1})_j\|_0\leq c_m.q_n^{k},$ for  $j=1,2.$ Hence  $\vertiii{\phi_n^m}_k\leq c_m(n,k)q_n^{k}$ , where $c_m$ is a constant and independent of $q_n.$ 
  Analogously, we have $ \vertiii{\phi_n^g}_k\leq c_g(n,k)q_n^{k}$ and $\vertiii{\phi_{n,i}}_k\leq c_i(n,k)q_n^{k}$ for $i \in \{0,\ldots,r-1\}$ where $c_g$ and $c_i$ are constants independent of $q_n$. With triangle inequality on the norm, we have  $\vertiii{\phi_n^w}_k \leq c_w(n,k) r. q_n^k.$
Using the above estimate and lemma (\ref{le:2a}), we have
\begin{align}
    \vertiii{\phi_n}_k &\leq c_k(n,k) \vertiii{\phi_n^g}_k^k.\vertiii{\phi_n^m\circ \phi_n^w}_k^k \nonumber\\
    &\leq  c_k(n,k) \vertiii{\phi_n^g}_k^k.\vertiii{\phi_n^m}_k^{k^2}.\vertiii{ \phi_n^w}_k^{k^2} \nonumber\\
    & \leq c_k(n,k) q_n^{2k^3+k} 
    %\leq C_2(n,k) q_n^{k^4}\nonumber
\end{align}
where $c_k(n,k)$ is a constant independent of $q_n$.
%{\bb For $\phi_n^{(w)}= \phi_{n,0}\circ \phi_{n,2}\circ...\phi_{n,r-1}$, we claim that $\|{D_a\phi_n^w}_j\|\leq c_0q_n^{rk}$. 
%The partial derivative with $|a|=k,$ consist of sums of products of atmost $r.k$ terms of the form 
%$$ (D_b[\phi_{n,i}]_j)\phi_{n,i+1} ....\phi_{n,r}
%$$ with $|b|\leq k$. Following the induction }
% Using the composition formula 
\end{proof}
%$R_n = \left[\frac{1}{2q_n}-\frac{1}{2\theta_nq_n},\frac{1}{2q_n}+\frac{1}{2\theta_nq_n} \right], \e_4>1/\theta_nq_n $}
%Remark: The map $\phi_n^m$ act as an identity on $Y_n.$ \\
\begin{figure}[t!]
    \centering
    \includegraphics[width=1\textwidth]{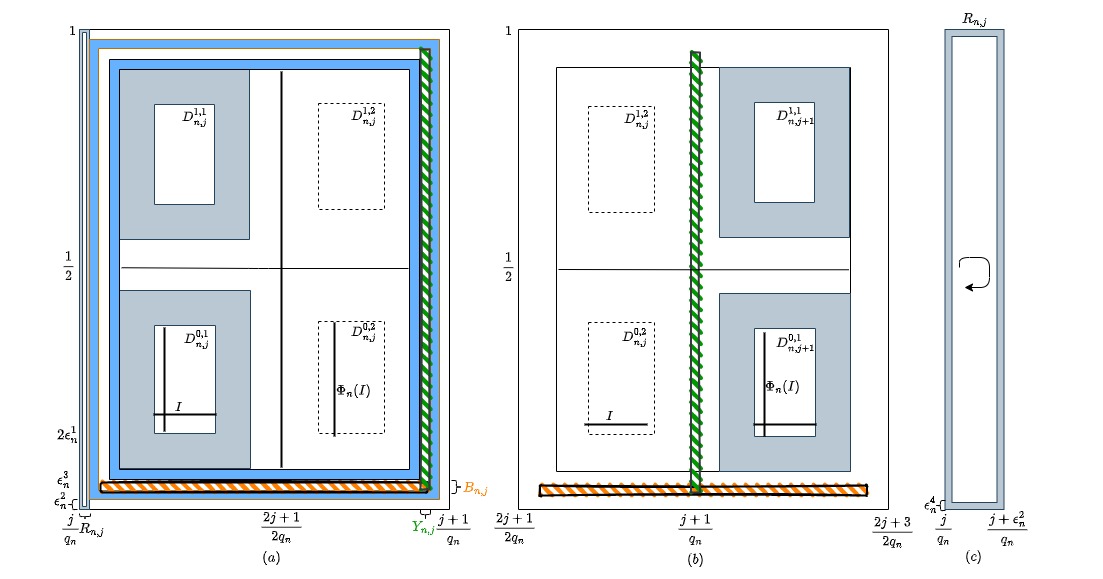}
    \caption{An example of action ${\phi}_n$ and $\Phi_n$ on the fundamental domains inside the $\T^2$ for $r=2$. The orange region, $B_{n,j}$, is transformed into the green region, $Y_{n,j},$ under the action of $\phi_n$. In (a), the horizontal line $I$ lying inside $D_{n,j}^{0,1}$ is transformed into vertical by $\phi_n^{-1}$ and then transferred to the right $D_{n,j}^{0,2}$ under the action of $\Phi_n$. Whereas in (b), the horizontal line $I$ lying inside $D_{n,j}^{0,2}$ is transferred to $D_{n,j+1}^{0,1}$ first and then transformed into vertical by $\phi_n$ under the action of $\Phi_n$. The same action of $\Phi_n$ will be followed inside regions $D_{n,j}^{1,1}$ and $D_{n,j}^{1,2}$ in both (a) and (b) respectively. In (c), the region inside $R_{n,j}$ is being rotated by the map $\phi_n$ by $\pi/2$.
    The blue and grey shaded region represent the error region for $\phi_n$.
    }
    \label{fig:1a}
\end{figure}
\subsection{The conjugation map \texorpdfstring{$h_n$}{Lg}} \label{sec:3a}
%For the following values of parameters : $\theta_n=q_n,  \d_n= \d_n' = \frac{1}{2^n\theta_nq_n}.$ 
The final conjugacy map $h_n:\T^2\longrightarrow\T^2$ is defined as a composition of the following maps as
\begin{align}\label{eq:3d}
h_n= g_n\circ \phi_n \circ P_n 
\end{align}where the diffeomorphism $P_n: \T^2\longrightarrow\T^2$ is defined by $P_n(x,y)= (x,y+ \kappa_n(x))$ with a smooth map
$\kappa_n: \T^1\longrightarrow \T^1$. For our specific situation, we choose  $\tilde{\kappa_n}:\left[0,\frac{1}{q_n}\right]\longrightarrow\T^1$ as follows, and then extend it $\frac{1}{q_n}$-periodically on the whole $\mathbb{T}^1,$ \begin{equation}\label{eqn:3.2}\tilde{\kappa_{n}}(x) =  \begin{cases}  \frac{2q_n}{n^2\e_n^{(2)}}x \ \   \  &,x\in [0,\frac{\e_n^{(2)}}{2q_n}] \\ -\frac{2q_n}{n^2\e_n^{(2)}}x+\frac{2}{n^2} \ \   \  &,x\in [\frac{\e_n^{(2)}}{2q_n},\frac{\e_n^{(2)}}{q_n}]\\0 \ \    &,x\in [\frac{\e_n^{(2)}}{q_n},\frac{1}{q_n}].   \end{cases} 
\end{equation}
%\begin{equation}\label{eqn:3.2}\tilde{\kappa_{n}}(x) =   \begin{cases} \frac{1}{n^2}\sin{\frac{2\pi q_nx}{2\e_n^2}},\ \   \  &x\in [0,\frac{\e_n^2}{q_n}) \\  0, \ \    &x\in [\frac{\e_n^2}{q_n},\frac{1}{q_n}]\end{cases} \end{equation}
Let $\kappa_n$ be the smooth approximation of $\tilde{\kappa_n}$ on $[0,1]$ by convolving it with a mollifier (Wikipedia, \url{ https://en.wikipedia.org/wiki/Mollifier}). Let $\rho$ be the standard mollifier on $\mathbb{R},$ and set $\rho(x)=\begin{cases}
  c\exp{\frac{1}{|x|^2-1}}\ \ \ &, |x|<1 \\
  0 \ \ \ &, \text{otherwise}
\end{cases}$, where $c$ is constant such that $\int_{\mathbb{R}} \rho(x)= 1$. Then, $$\kappa_n(x)= \lim_{\delta\rightarrow 0} \kappa_n^{\delta}(x)= \lim_{\delta\rightarrow0} \delta^{-1}\int_{\mathbb{T}^1}\rho\left(\frac{x-y}{\delta}\right)\tilde{\kappa}_n(y)dy.$$

\remark  The map $\kappa_n(x)= q_n x$ on $\T^1$ is considered in \cite{FSW} to control almost all the orbits of space. 
\remark For minimality, the orbit of every point has to be dense. The map $\phi_n^m$ takes care of all the points inside $\T^2$ except the points whose whole orbit gets trapped inside the Error zone(where we do not have any control), $\mathbb{E}_n^m$, of $\phi_n^m$. The map $P_n$ acts as the vertical translation such that such an orbit would enter the minimality zone, and no whole orbit of a point gets trapped inside the Error zone. Also, note that $P_n$ acts as an identity outside the region $\cup_{i=0}^{q_n-1}R_{n,i}.$ 

\remark\label{re:3.3a} Also note that $\|D^k\kappa_n\|_0 \leq \max\limits_{x\in [-1,1]}|\tilde{\kappa}_n|.\|D^k\rho\|_0 \leq (\frac{2k\sqrt{18}}{\e})^{2k}.k!. q_n^{k}.$ 
%$P_n(x,y)= (x,y+x.q_n)$ and the maps $\phi_n, g_n$ is defined in ${\rr (\ref{eq:3a}), cite } $ respectively.
%\remark The map $P_n$ allows to control on the all orbits for........

\section{Convergence}
There are some standard results on the closeness between the maps constructed as the conjugation of translations on the torus. The following two lemmas are identical to lemma 3,4 in \cite{FSW} with minor to no modification; hence, we skip the proofs for brevity.
%hence we relegate the proof to appendix
\begin{lemma} \label{la:1a}
Let $k\in \mathbb{N}$. For all $\alpha,\beta \in \mathbb{R}$ and all $h\in \text{Diff}^{\infty}(\mathbb{T}^2)$, we have the estimate
$$d_k(h S_{\alpha}h^{-1},h S_{\beta}h^{-1}) \leq C_k \max\{\vertiii{h}_{k+1},\vertiii{h^{-1}}_{k+1}\}|\alpha- \beta|,$$
where $C_k$ is a constant that depends only on k. 
\end{lemma}
%In this section we prove the convergence of Approximation by conjugation method and evaluate some estimates that are important to proof of the theorems. Next lemma is identical to lemma 3,4 in \cite{FSW} with minor to no modification and hence we relegate the proof to appendix.
\begin{lemma}  \label{la:1b}
For any $\epsilon>0,$ let $k_n$ be a sequence of natural numbers satisfying $\sum\limits_{n=1}^{\infty}\frac{1}{k_n}<\epsilon$. Suppose for any Liouville $\alpha$, there exist a sequence of rationals $\{\alpha_n\}$ that satisfy:
\begin{equation}\label{eqn:5.5}
    |\alpha-\alpha_n|<\frac{1}{2^{n+1}k_nC_{k_n}q_n \vertiii{H_n}_{k_{n+1}}^{k_{n+1}}}
\end{equation}
    where $C_{k_n}$ is the same constant as in lemma \ref{la:1a}. Then
the sequence of diffeomorphisms  $T_{n} = H_{n} \circ S_{\alpha_{n+1}}\circ H_{n}^{-1}$ converges to  $T\in \text{Diff}^{\infty}(\mathbb{T}^2, \mu)$ in the $C^{\infty}$ topology. Moreover, for any $m\leq q_{n+1},$ we have 
\begin{equation}\label{eq:4a}
d_0(T^m,T_{n}^m) \leq \frac{1}{2^{n+1}},
\end{equation} 
\end{lemma}
%Following on the line of Windsor and \cite{Ba-Ku} of proposition 3 in  \cite{FSW}. {\rr this sentence does not make sense}
\begin{lemma}
For any $k\in \mathbb{N},$ the conjugating diffeomorphism defined in $(\ref{eq:3a})$ and $(\ref{eq:3d})$ satisfy the following norm estimates as
\begin{enumerate}
    \item $\vertiii{h_n}_k\leq c_k(n,k).q_n^{2k_n^4+2k^2}$ , where $c_k(n,k)$ is constant independent of $q_n.$
    \item $\vertiii{H_n}_k \leq \hat{c}_k(n,k). q_n^{2k_n^5+2k^3+ k}$, where $\hat{c}_k(n,k)$ is constant independent of $q_n.$
    \item For $\alpha$ Liouville, there exist a sequence of rational $\{\alpha_n\}$ satisfying  {\eqref{eqn:5.5}}.
\end{enumerate}
\end{lemma}

\begin{proof}
The map $h_n$ is defined by  
\begin{align}
h_n(x,y)&= g_n\circ \phi_n \circ P_n(x,y)\nonumber\\
        &= ([\phi_n(x,y+\kappa_n(x))]_1 + \lfloor{nq_n}^{\sigma}\rfloor[\phi_n(x,y+\kappa_n(x))]_2 , [\phi_n(x,y+\kappa_n(x))]_2 ) \nonumber
\end{align}
By lemma \ref{le:2a} and remark \ref{re:3.3a}, we have estimate:
\begin{align}
    \vertiii{h_n}_k &\leq 2.(nq_n)^{k-1}.\vertiii{\phi_n}_k^k.\vertiii{\kappa_n}_k^k  \nonumber\\
    &\leq c_k(n,k).q_n^{2k_n^4+2k^2}\nonumber
\end{align}
Similarly, $\vertiii{H_n}_k= \vertiii{H_{n-1}\circ h_n}_k \leq \vertiii{H_{n-1}}_k^k\vertiii{h_n}_k^k $. Since the derivatives of $H_{n-1}$ of $k$th order is independent of $q_n$, we can conclude 
$\vertiii{H_{n}}_k \leq  \hat{c}_k(n,k)q_n^{2k_n^5+2k^3+ k}.$\\
For $\alpha$ being a Liouville, we can choose a sequence of rationals $\alpha_n=\frac{p_n}{q_n}$($p_n, q_n$ are coprime) that satisfy the following property: 
\begin{align}
|\alpha-\alpha_n|& \leq \frac{1}{2^{n+1}k_nC_{k_n}q_n^{{2(k_n+1)^5+2(k_n+1)^3+(k_n+1)}}} \nonumber\\
&\leq \frac{1}{2^{n+1}k_nC_{k_n}q_n\vertiii{H_n}_{k_{n+1}}^{k_{n+1}}} \nonumber
\end{align} 
\end{proof}
\remark Finally, we have proven the estimate on the norms of the  conjugation map $H_n$ as in \cite{FS}. Also, the existence of rationals satisfying (\ref{eqn:5.5}) guarantees the convergence of $T_n$ to $T\in \text{Diff}^{\infty}(\T^2,\mu)$ in lemma {\ref{la:1b} }. 
\section{Weak mixing, Minimality and Generic points}
To prove theorem A which needs a couple of preliminary results.
\subsection{A Fubini criterion for weak mixing}
Here we state a few definitions and the criterion for weak mixing described in \cite{FS}  for $\mathbb{T}^2.$  
\begin{definition} A collection of disjoint sets $\eta_n$ on $\T^2$ is called partial decomposition of $\T^2$. A sequence of partial decompositions $\eta_n$ converges to the decomposition into points (notation: $\eta_n \rightarrow \varepsilon$) if, any measurable set A, for any $n$ there exists a measurable set $A_n$, which is a union of elements of $\eta_n$, such that $\lim_{n\rightarrow \infty} \mu(A \triangle A_n) = 0$ (here $\triangle$ denotes the symmetric difference).
%Let $\eta_n$ be a partial decomposition of $\T^1$ into intervals and consider on $\T^2$ thedecomposition $\eta$ consisting of intervals in $\hat{\eta}$ times some $r\in [0, 1]$. Sets of this form will be calledhorizontal intervals and decompositions of this type standard partial decompositions. On the other hand, sets of the form ${\theta} × J$, where J is an interval on the r-axis, are called verticalinterval
\end{definition}
%weak mixing definition here,$$ {\left\lvert \mu(B\cap f^{-m_n}(A)) - \mu(B)\mu(A) \right\rvert} \longrightarrow 0.$$
Recall the notion of $(\gamma, \delta,\epsilon)$-distribution of a horizontal    interval in the vertical direction.
\begin{definition} ($(\gamma, \delta,\epsilon)$- distribution):- A diffeomorphism $\Phi: M \longrightarrow M, \ (\gamma, \delta,\epsilon)$ distributes a horizontal interval $I \in \eta,$ where $\eta$ is the partial decomposition of $M$ (or $\phi(I) $ is $(\gamma, \delta,\epsilon)$- distributed on $M$ ), if
\begin{itemize}
    \item $J=\pi_{y}(\Phi(I))$ is an interval  with $1-\delta \leq \lambda(J) \leq 1,$ where $\pi_y$ is the projection map onto the y coordinate. 
    \item $\Phi(S) \subseteq K_{c,\gamma}= {[c,c+\gamma]\times J}$ for some c (i.e $\Phi(S)$ is almost vertical);
    \item for any interval $\tilde{J} \subseteq J$ we have:
    ${\left\lvert \frac{\lambda(I\cap \Phi^{-1}(\mathbb{T}\times \tilde{J}))}{\lambda(I)} - \frac{\lambda(\tilde{J})}{\lambda(J)}\right\rvert} \leq \epsilon \frac{\lambda(\tilde{J})}{\lambda(J)}.$
   %  we will write the above relation as 
    % $${\left\lvert \lambda(I\cap \Phi^{-1}(\mathbb{T}\times \tilde{J})){\lambda(J)} - {\lambda(I)}{\lambda(\tilde{J})}\right\rvert} \leq \epsilon {\lambda(I)}{\lambda(\tilde{J})}.$$
\end{itemize}
\end{definition}
%\begin{lemma} Let $D:\mathbb{T}^2\longrightarrow\mathbb{T}^2$ be map defined as $D(x,y)= (x,y+qx),$ where $q\in \mathbb{N}$ is constant.
%Suppose that a diffeomorphism $\phi: \mathbb{T}^2\longrightarrow \mathbb{T}^2$ is $(\gamma, \delta,\epsilon)$- distribute a horizontal interval. Then the composition $D\circ \phi$ is $(\gamma, \delta,\epsilon)$- distribution.
%\end{lemma}
%\begin{proof}
%Since $D$ introduce the shear in the vertical direction only, therefore $D\circ \phi$  satisfy all the condition of  $(\gamma, \delta,\epsilon)$- distribution. 
%\end{proof}

\begin{prop}[\cite{FS}, Proposition 3.9] \label{eq:pr1} Assume $T_n = H_n \circ S_{{\alpha_n+1}}\circ H_n^{-1}$ is the sequence of diffeomorphism constructed by (\ref{eq:3b}), (\ref{eq:3c}) and (\ref{eq:3d}) such that all $n,$ $\|DH_{n-1}\|_0 < \ln q_n $ holds.
%Assume that  is sequence of diffeomorphism constructed with following
%\begin{align*}
%H_n = h_1 \circ \dots  \circ h_n, \
%h_n \circ S_{\alpha_n} = S_{\alpha_n} \circ h_n \textit{where}\\h_n = g_n \circ \phi_n,  ({\rr check this condition.})\end{align*}These property holds.$$g_n(\theta, r) = (\theta+[nq_n^\sigma]r,r)\textit{with some} 0<\sigma <\frac{1}{2}$$
Suppose $\lim_{n\to\infty}T_n = T$ exists. If there exists a sequence of natural numbers $\{\mathfrak{m}_n\}$ such that $d_o(f^{\mathfrak{m}_n}, f_n^{\mathfrak{m}_n})<\frac{1}{2^n}$,    and a sequence of standard partial decomposition $\eta_n$ of $M$ into horizontal intervals of length less than $\frac{1}{q_n}$ satisfying
\begin{itemize}
	\item[1.] $\eta_n\to \varepsilon$
	\item[2.] for $I_n\in \eta_n$, the diffeomorphism $\Phi_n = \phi_n \circ S_{\alpha_n+1}^{\mathfrak{m}_n} \circ \phi_n^{-1}$ is
	$(\frac{1}{nq_n^\sigma},\frac{1}{n},\frac{1}{n})$ uniformly distribute the interval $I_n$.
	\end{itemize}
Then limiting diffeomorphism  $T$ is weak mixing.
\end{prop}

\subsection{Proof for weak mixing}\label{sec:5.1:b}
The specific scheme that we describe here builds on the construction in \cite{FS}. First, consider a subset of the $\mathbb{T}^2$ as  \begin{align}\label{eq:5a}
\E_n^w=\left(\bigcup\limits_{k=0}^{2q_n-1}\left[\frac{k}{2q_n}-\frac{2\e_n^{(1)}}{q_n},\frac{k}{2q_n}+\frac{2\e_n^{(1)}}{q_n}\right] \times \mathbb{T}^1\right)\bigcup\left(\bigcup\limits_{t=0}^{r-1} \mathbb{T}^1 \times \left[\frac{t}{r}-2\e_n^{(1)}, \frac{t}{r}+2\e_n^{(1)} \right]\right).
\end{align}
\subsubsection{Action of \texorpdfstring{$\phi_n$}{Lg}}
Consider the interval, $I_{n,j} \subseteq D_{n,j}^{t,1}$ for some fixed $t$ and $j$ of the form $I_{n,j} = I_{n,j}^0 \times \{s\}$ where $s \in \left[\frac{t}{r},\frac{t+1}{r}\right]$ and
\begin{align}\label{eq:5.1a}
I_{n,j}^0 = \left[\frac{j}{q_n} + \frac{2}{3nq_nr},\frac{j}{q_n} + \frac{1}{2q_n}- \frac{2}{3nq_nr}\right] 
\end{align}
From our construction of $\phi_n$, the image of $I_{n,j}$ under both $\phi_n$ and $\phi_n^{-1}$ is an interval of type $\{\theta \} \times\left[\frac{t}{r}+\frac{2}{3nr},\frac{t+1}{r} -\frac{2}{3nr} \right]$ for some $\theta \in I_{n,j}^0.$
\subsubsection{Choice of \texorpdfstring{$\mathfrak{m}_n$}{Lg}- mixing sequence}\label{sec:5.1a}%{\bb Change $m_n$ to $\mathfrak{m}_n,$ note $\mathfrak{m}_n$ had been used in other decomposition}
Consider  $\mathfrak{m}_n= \min\left\{m\leq q_{n+1} \ | \ \inf_{k\in \Z} \left \lvert m\frac{q_np_{n+1}}{q_{n+1}}-\frac{1}{2}+k \right\rvert \leq \frac{q_{n}}{q_{n+1}}\right\}$
and $\mathfrak{a}_n = (\mathfrak{m}_n\alpha_{n+1}- \frac{1}{2q_n}\mod\frac{1}{q_n})$
as defined in Fayad's paper for the torus case and with the growth assumption, $q_{n+1} > 10n^2q_n$ would result here: $$|\mathfrak{a}_n| \leq \frac{1}{q_{n+1}} \leq \frac{1}{10n^2q_n}.$$
Further, if we define a precise domain as
${\overline{D}_{n,j}^{t,1}} = I_{n,j}^0\times \left[\frac{t}{r},\frac{t+1}{r}\right] \subset D_{n,j}^{t,1}$ for some $j\in \mathbb{Z}$, then we would have 
$S_{\alpha_{n+1}}^{\mathfrak{m}_n}(\overline{D}_{n,j}^{t,1}) \subset D_{n,j'}^{t,2}$ for some $j'\in \mathbb{Z}.$

\subsubsection{Choice of decomposition \texorpdfstring{$\eta_n^t$}{Lg}}\label{sec:5.1b}
For fixed $t\in\{0,1,\ldots,r-1\}$, we consider the partial decomposition $\eta_n^t$ of the set $N^t$,
%\backslash\mathcal{E}_n$ 
outside $\E_n^w$, which consists of two types of horizontal intervals: $I_{n,j}=I_{n,j}^0 \times \{s\} \subset D_{n,j}^{t,1}$ and $\overline{I}_{n,j}= \overline{I}_{n,j}^0 \times \{s'\} \subset D_{n,j}^{t,2} $ where $s, s' \in \left[\frac{t}{r},\frac{t+1}{r}\right]$, and $I_{n,j}^0$ by (\ref{eq:5.1a}), and 
\begin{align}
   % I_{n,j}^0 &= \left[\frac{j}{q_n} + \frac{1}{6nq_n},\frac{j}{q_n} + \frac{1}{2q_n} - \frac{1}{6nq_n}\right] \label{eq:01}
\overline{I}_{n,j}^0 &= \left[\frac{j}{q_n} + \frac{1}{2q_n} - \frac{2}{3nq_nr} -\mathfrak{a}_n,\frac{j+1}{q_n} - \frac{2}{3nq_nr}-\mathfrak{a}_n\right]. \label{eq:02}
\end{align}
%$\eta_n^t=\{ I_{n,j}, \overline{I}_{n,j}$
Note that for any element $I _n\in \eta_n^t,$ we have $\pi_y(\phi_n(I_n))\subset \left[\frac{t}{r},\frac{t+1}{r}\right].$ Since the length of intervals goes to zero and  $\sum_{I_n\in \eta_n}\lambda(I_n)\leq 1 - \lambda(\E_n^w) \leq 1 - \frac{4}{n}\rightarrow 1,$ it implies $\eta_n^t\rightarrow 0$ as $n\rightarrow\infty.$    
\begin{lemma}\label{le:5.1b}
For any $t\in \{0,1,...,r-1\}$. The map  $\Phi_n = \phi_n\circ P_n\circ S_{\alpha_{n+1}}^{\mathfrak{m}_n} \circ P_n^{-1}\circ \phi_n^{-1}$ transform the elements of the partial decomposition, i.e. $I_{n,j}= I_{n,j}^0\times \{s\} \in \eta_n^t$, into vertical interval of the form $\{\theta\} \times[\frac{t}{r}+\frac{2}{3nr},\frac{t+1}{r}-\frac{2}{3nr}]$ for some $\theta \in I_{n,j}^0$(see Figure-1).
\end{lemma}

\begin{proof}
From our construction of $\phi_n\circ P_n$, an interval $I_{n,j}= I_{n,j}^0\times \{s\} \subset D_{n,j}^{t,1}$ where $s\in \left[\frac{t}{r}+\frac{2}{3nr},\frac{t+1}{r} -\frac{2}{3nr} \right]$, we have $P_n^{-1}\circ \phi_n^{-1}(I_{n,j})= \{\theta \} \times\left[\frac{t}{r}+\frac{2}{3nr},\frac{t+1}{r} -\frac{2}{3nr} \right]$ for some $\theta\in I_{n,j}^0.$\\
With the specific choice of sequence $\mathfrak{m}_n$ and the condition mentioned in section (\ref{sec:5.1a}), we get 
 $$ S_{\alpha_{n+1}}^{\mathfrak{m}_n}\circ P_n^{-1}\circ \phi_n^{-1}(I_{n,j}) = \{\theta'\}\times \left[\frac{t}{r}+\frac{2}{3nr},\frac{t+1}{r} -\frac{2}{3nr} \right] \subset D_{n,j'}^{t,2},$$ for some $\theta'\in\mathbb{T}$ and $j' \in \mathbb{Z}$. Since $\kappa_n$ acts as an identity on $[\frac{\e_n^{(2)}}{q_n}, \frac{1}{q_n}]$ and 
 the fact $\phi_n$ acts as an identity on $D_{n,j'}^{t,2},$  concludes the claim.
 Similarly, for the interval $\overline{I}_{n,j} = \overline{I}_{n,j}^0 \times \{s\} \subset D_{n,j}^{t,2},$ we deduced that $$\phi_n\circ P_n\circ S_{\alpha_{n+1}}^{\mathfrak{m}_n} \circ P_n^{-1}\circ \phi_n^{-1}(\overline{I}_{n,j}) = \{\theta' \} \times \left[\frac{t}{r}+\frac{1}{3nr},\frac{t+1}{r} -\frac{1}{3nr} \right]  \subset D_{n,j'}^{t,1}$$ for some $j'\in \mathbb{Z}$ and $\theta'\in \mathbb{T}.$
\end{proof}

\subsection{A criterion for minimality}
The aim of this section is to deduce a criterion for minimality for our explicit construction.  Precisely, it allows us to understand the action $\phi_n$ on the region $R_{n,i}$ explained in step 3, section \ref{eq:sec1}. 
Here, we define the following partition of set $R_{n,i}$ excluding the set $\E_n^m$, for any natural number $l_n$, as follows
%,= \left[\frac{i}{2q_n}-\frac{1}{2\theta_nq_n},\frac{i}{2q_n}+\frac{1}{2\theta_nq_n} \right] \times \mathbb{T}^1.$ 
%For notation convenience, denote $\gamma_{n,i}= \frac{i}{2q_n}-\frac{1}{2\theta_nq_n},$ then  $R_{n,i}= \left[\gamma_{n,i},\gamma_{n,i}+\frac{1}{\theta_nq_n} \right] \times \mathbb{T}^1.$ For the parmeter $l_n\in \mathbb{N},$ consider the following partition of $R_{n,i}$ which we need in the course of proof: 
%The details of the construction
%Introduce the following partition of region $R_{n,i}=\left[\frac{i}{q_n}+\d_n', \frac{i}{q_n}+ \frac{1}{\theta_nq_n}\right)\times \mathbb{T}^1$:
%\A_{n} :=\left\{A_{i,k}^n=\left[\gamma_{n,i}+ \frac{k}{l_n\theta_n q_n}, \gamma_{n,i}+ \frac{k+1}{l_n\theta_nq_n}\right)\times \mathbb{T}^1\ \ : i= 1,3,...,\left\lfloor\frac{2q_n-1}{2}\right\rfloor, \ k= 0,1,...,l_n-1 \right\} \nonumber
\begin{align}
A_{i,k}^n &:= \left[\frac{i}{q_n}+\frac{2\e_n^{(4)}}{q_n}+ \frac{k(\e_n^{(2)}-4\e_n^{(4)})}{l_nq_n}, \frac{i}{q_n}+ \frac{2\e_n^{(4)}}{q_n}+ \frac{(k+1)(\e_n^{(2)}-4\e_n^{(4)})}{l_nq_n}\right)\times \left[2\e_n^{(4)},1- 2\e_n^{(4)} \right] \nonumber\\
% \B_{n} :=\left\{B_{i,k}^n :=  \left[\gamma_{n,i}, \gamma_{n,i}+ \frac{1}{\theta_nq_n}\right)\times \left[\frac{k}{l_n}, \frac{k+1}{l_n}\right) \ \ : i= 1,3,...,\left\lfloor\frac{2q_n-1}{2}\right\rfloor, \  k= 0,1,... ,l_n-1 \right\} \nonumber
B_{i,k}^n &:= \left[\frac{i}{q_n}+\frac{2\e_n^{(4)}}{q_n}, \frac{i}{q_n}+  \frac{\e_n^{(2)}}{q_n}-\frac{2\e_n^{(4)}}{q_n}\right)\times \left[2\e_n^{(4)}+ \frac{k(1-4\e_n^{(4)})}{l_n},2\e_n^{(4)}+ \frac{(k+1)(1-4\e_n^{(4)})}{l_n}\right]\nonumber.
\end{align} 
Let's denote the family of these subsets by $\A_n=\{A_{i,k}^n,  \ \ i= 0,\ldots,q_n-1, \ k= 0,\ldots ,l_n-1\}$ and $\B_n= \{B_{i,k}^n,  \ \  i= 0,\ldots,q_n-1, \ k= 0,\ldots ,l_n-1\}$. 
\remark Note that under the transformation $\phi_n$, the elements of $\A_n$ map to the elements of $\B_n.$ In particular, by (\ref{eqn:3.1.2}), we get $\phi_n^m(A_{i,k}^n)= B_{i,k}^n$ for all $i,k$ as defined above.
%$ i \in \{1,3,...,\left\lfloor\frac{2q_n-1}{2}\right\rfloor\}$ 
Since $R_{n,i}$ lies inside $\Sigma_1$ and the maps $\phi_{n}^w,\phi_{n}^g$ act as an identity on $\Sigma_1$. Therefore $\phi_n(A_{i,k}^n)= B_{i,k}^n$
%nd  $P_n(A_{i,k}^n)= A_{i,k}^n$ for all $i,k$ as described above. 
%{\bb change $r_n= \theta_n, \epsilon_n'= \d_n'$ try to change the notation $\phi_n= \Phi_n$}
\begin{lemma}
Let $x\in \T^2$ and $q_{n+1}>l_nq_n^2$ be arbitrary, the orbit $\{S_{\alpha_{n+1}}^k(x)\}_{k=0}^{q_{n+1}-1}$ intersects every set $P_n^{-1}(A_{i_1,i_2}^n)$.% $\forall \  i_1\in \{0,\ldots, q_n-1\}, i_2\in\{0,\ldots, k_n-1\}.$
\end{lemma} %{\rr mention that you are choosing $q_{n+1}$ in the proof}
\begin{proof}
Let fix $x=(x_1,x_2)\in \T^2$, $i_1\in\{0,\ldots, q_n-1\}$ and $i_2\in\{0,\ldots, l_n-1\}.$
The map $P_n$ acts as the vertical translation on $\T^2$ and with the choice of $\kappa_n$ function (see (\ref{eqn:3.2})),
the set $A_{i_1,i_2}^n$ under the map $P_n^{{-1}},$ 
$c\leq \pi_y(P_n^{-1}(A_{i_1,i_2}^n))\leq c+ \gamma,$ where $c=\frac{2\e_n^{(4)}}{q_n}+ \frac{i_2(\e_n^{(2)}-4\e_n^{(4)})}{l_nq_n}$ and $\gamma= \frac{(\e_n^{(2)}-4\e_n^{(4)})}{n^2\e_n^{(2)}l_n}.$ Since $[2\e_n^{(4)},1-2\e_n^{(4)}]\subseteq \pi_y(A_{i_1,i_2}^n)$, it satisfy $\pi_y(P_n^{-1}(A_{i_1,i_2}^n))=\T^1.$\\
Since $\{k\alpha_{n+1}\}_{k=0,1,\ldots,q_{n+1}-1}$ is equidistributed on $\T^1$ and $S_{\alpha_{n+1}}$ act as horizontal translation on $\T^2$, therefore there exist $k \in \{0,1,\ldots,q_{n+1}-1\}$ such that
$S_{\alpha_{n+1}}^{k}(x)\in P_n^{-1}(A_{i_1,i_2}^n
),$  in other words, there exist $k \in \{0,1,\ldots,q_{n+1}-1\}$ such that
$x_1+k\alpha_{n+1}\in \pi_x(P_n^{-1}(A_{i_1,i_2}^n))$ and $x_2\in \pi_y(P_n^{-1}(A_{i_1,i_2}^n)).$ 
\end{proof}
\begin{prop} \label{pr:2a}
\begin{enumerate}
%[label=\roman*]
    \item For every $z\in \T^2$, the iterates $\{\phi_{n}\circ P_n \circ S_{\alpha_{n+1}}^k\circ H_{n}^{-1}(z)\}_{k=0,1,\ldots,q_{n+1}-1}$ meets every set of the form $\left[\frac{i}{q_n},\frac{i+1}{q_n}\right]\times \left[\frac{j}{l_n},\frac{j+1}{l_n}\right]$, where $l_n\in \N$ and satisfy (\ref{eq:l_n}).
    %\item Show that the iterates $\{\phi_{n}\circ P_n \circ S_{\alpha_{n+1}}^i\circ H_{n}^{-1}(z)\}_{i=0,1,..q_{n+1}-1}$ meets every set of the form $\left[\frac{i}{n},\frac{i+1}{n}\right]\times \left[\frac{j}{n},\frac{j+1}{n}\right]$ for every $z\in \mathbb{T}^2.$ {\rr rephrase language}
    \item Suppose the sequence of diffeomorphism $T_{n} = H_{n}\circ S_{\alpha_{n+1}}\circ H_{n}^{-1}$ converges to $T\in \text{Diff}^{\infty}(\mathbb{T}^2,\mu)$ in the $C^{\infty}$ topology and satisfies the proximity condition, $d_0(T_{n}^k,T^k)<\frac{1}{2^n} \ \forall k= 0,\ldots,q_{n+1}-1$, then the limiting diffeomorphism  $T$ is  minimal. 
\end{enumerate}
\end{prop}
%if the sequence of diffeomorphism $T_{n+1}= H_{n+1}\circ S_{\alpha_{n+1}}\circ H_{n+1}^{-1}$ converges to $T\in Diff^{\infty}(\mathbb{T}^2,\mu)$ in the $C^{\infty}$ topology and satisfies the proximity $d_0(T_{n+1}^i,T^i)<\frac{1}{2^n} \ \forall i= 0,...q_{n+1}-1$. %Then limiting diffeomorphism  $T=\lim_{n\longrightarrow\infty} T_{n+1},$ then $T$ is minimal.
\begin{proof}
Let $x\in \mathbb{T}^2$ and $i\in \{0,1,\ldots,q_n-1\}$ and $j\in \{0,1,\ldots,l_n-1\}$  be arbitrary.
Note that if $\alpha_{n+1}$ is chosen large enough that $q_{n+1}>{l_nq_n^2}$ and by above lemma, there exist $k \in \{0,1,\ldots,q_{n+1}-1\}$ such that 
$S_{\alpha_{n+1}}^{k}(x)\in P_n^{-1}(A_{i,j}^n
).$ 
%{\bb $S_{\alpha_{n+1}}^{k}(x)\in P_n^{-1}\left(\Big[\frac{i'+\epsilon_n}{q_n}+\frac{j'}{n^2r_nq_n},\frac{i'+\epsilon_n}{q_n}+\frac{j'+1}{n^2r_nq_n} \Big) \times \mathbb{T}^1 \right).$}
Under the conjugation map, we have 
\begin{align}
\phi_{n}\circ P_n \circ S_{\alpha_{n+1}}^k(x) &\in \phi_n(A_{i,j}^n)= B_{i,j}^n ;\nonumber  \\
B_{i,j}^n &\subset \left[\frac{i}{q_n},\frac{i+1}{q_n} \right]\times \left[\frac{j}{l_n},\frac{j+1}{l_n}\right]\label{eq:m2}
\end{align}
It shows that for $x=H_{n}^{-1}(z)$,
the orbit $\{\phi_n\circ P_{n}\circ S_{\alpha_{n+1}}^k\circ H_{n}^{-1}(z)\}_{k=0,1,\ldots,q_{n+1}-1}$ meets every set of type $\left[\frac{i}{q_n},\frac{i+1}{q_n}\right]\times \left[\frac{j}{l_n},\frac{j+1}{l_n}\right].$ 
Also, record that the collection of such sets $\left[\frac{i}{q_n},\frac{i+1}{q_n}\right]\times \left[\frac{j}{l_n},\frac{j+1}{l_n}\right]$ for $0\leq i< q_n,\ 0\leq j<l_n$ covers the whole space $\mathbb{T}^2$ and 
$$ diam\left(H_{n-1}\circ g_n\left(\left[\frac{i}{q_n},\frac{i+1}{q_n}\right]\times \left[\frac{j}{l_n},\frac{j+1}{l_n}\right]\right)\right)\leq \|DH_{n-1}\|_0.\|Dg_n\|.\frac{2}{l_n}$$ which goes to 0 as $n \rightarrow\infty$(by condition(\ref{eq:l_n})). Hence, for $\varepsilon>0$ and $y\in \mathbb{T}^2$ there is $n_1\in \mathbb{N}:$ there exist a set 
$$H_{n-1}\circ g_n \left(\left[\frac{i}{q_n},\frac{i+1}{q_n}\right]\times \left[\frac{j}{l_n},\frac{j+1}{l_n}\right]\right) \subset B_{\frac{\varepsilon}{2}}(y) \   \forall \  n> n_1$$
For $H_n= H_{n-1}\circ g_n\circ \phi_n \circ P_n$,
we use the condition of convergence for diffeomorphism $\{T_n\}$ and $d_0(T_{n}^k,T^k)<\frac{1}{2^n}$. Hence, we can conclude that for arbitrary $x, y\in \mathbb{T}^2$ and $\varepsilon>0,$ there exist $n_2\in \mathbb{N}$ such that $d_0(T_{n}^k,T^k)<\frac{\varepsilon}{2} \  \forall  \ k=0,...q_{n}-1; n> n_2.$
Assuming $n> \max\{n_1,n_2\},$ there is a set $H_{n-1}\circ g_n\left(\left[\frac{i}{q_n},\frac{i+1}{q_n}\right]\times \left[\frac{j}{l_n},\frac{j+1}{l_n}\right]\right) \subset B_{\frac{\varepsilon}{2}}(y)$ and $T_{n}^k(x)\in H_{n-1}\circ g_n\left(\left[\frac{i}{q_n},\frac{i+1}{q_n}\right]\times \left[\frac{j}{l_n},\frac{j+1}{l_n}\right]\right) \subset B_{\frac{\varepsilon}{2}}(y)$ for some $k<q_{n+1}$. 
With the triangle inequality, we have 
\begin{align}
d(T^k(x),y) &\leq d(T^k(x),T_{n}^{k}(x))+ d(T_{n}^{k}(x),y)\nonumber \\
&\leq d_0(T^k,T_{n}^k) + \frac{\varepsilon}{2} < \varepsilon. \nonumber
\end{align}
i.e. $T^k(x)\in B_{\varepsilon}(y)$  and which implies T is minimal. 
\end{proof}

\subsection{ A Generic Measure}
The following results allow us to show the existence of the generic points residing inside the region ${\mathcal{G}}_{n}= \cup_{i=0}^{q_n-1}B_{n,i}$. Denote $\mathcal{Y}_n=\cup_{i=0}^{q_n-1}Y_{n,i}$ (defined in section \ref{eq:sec1}, step 2) and $\mathcal{D}_{n}=\T^2$.
%allows us to produce generic points  genericity of a measure which is supported on the region ${\mathcal{G}}_{\delta_n}= \T^1 \times  \left[\frac{-\delta_n}{2},\frac{+\delta_n}{2}\right].$ 
First, we introduce the following partitions of the sets $\mathcal{G}_{n}, \mathcal{Y}_n$ and $\mathcal{D}_{n}$ for any natural number sequence $s_n>q_n$, by the family of subsets ${G}_{i,j}^n$, ${Y}_{i,j}^n$, and ${\Delta}_{i,j}^n$  respectively, for $0\leq i< q_n,\  0\leq j< s_n$:  
\begin{align}
{G}_{i,j}^n &:=\left [\frac{i}{q_n} + \frac{2\e_n^{(2)}}{q_n}+\frac{j(1-4\e_n^{(2)})}{s_nq_n} , \frac{i}{q_n}+\frac{2\e_n^{(2)}}{q_n}+\frac{(j+1)(1-4\e_n^{(2)})}{s_nq_n}\right] \times \left[2\e_n^{(2)},\e_n^{(3)} \right]\nonumber\\
{Y}_{i,j}^n &:=\left [\frac{i}{q_n} + \frac{1-\e_n^{(3)}}{q_n}, \frac{i}{q_n} + \frac{1-2\e_n^{(2)}}{q_n}\right] \times \left[2\e_n^{(2)} + \frac{j(1-4\e_n^{(2)})}{s_n}, 2\e_n^{(2)} + \frac{(j+1)(1-4\e_n^{(2)})}{s_n}\right]\nonumber\\
{\Delta}_{i,j}^n &:=\left [\frac{i}{q_n}  ,\frac{i+1}{q_n}\right] \times \left[\frac{j}{s_n}, \frac{j+1}{s_n} \right]. \label{eq:5.3a}
%{\mathcal{G}}_{n}  &:=\left\{{G}_{i,j}^n :=\left [\frac{i}{q_n} +\frac{j}{s_nq_n} ,\frac{i}{q_n}+\frac{j+1}{s_nq_n}\right) \times \left[\frac{-\delta_n}{2} , \frac{+\delta_n}{2} \right) \ \ : 0\leq i< q_n, 0\leq j< s_n \right\} \nonumber \\
%\mathcal{Y}_{n}  &:=\left\{{Y}_{i,j}^n :=\left [\frac{i}{q_n},\frac{i}{q_n}+ \delta_n \right) \times \left[\frac{j}{s_n},\frac{j+1}{s_n} \right)  \ \ : 0\leq i< q_n, 0\leq j< s_n \right\}\\
%{\mathcal{D}}_{n}  &:=\left\{{\Delta}_{i,j}^n :=\left [\frac{i}{q_n}  ,\frac{i+1}{q_n}\right) \times \left[\frac{j}{s_n}, \frac{j+1}{s_n} \right) \ \ : 0\leq i< q_n, 0\leq j< s_n \right\} \label{eq:5.3a}
\end{align}
%Consider $\widetilde{\delta}_n=\frac{\delta_n}{2^{s_nq_n}},$ and another subset ${\mathcal{G}}_{\widetilde{\delta}_n} = \T^1 \times \left[\frac{-\widetilde{\delta}_n}{2} , \frac{+\widetilde{\delta}_n}{2} \right)$ corresponding to the value $\widetilde{\delta}_n.$
%For $x=(0,0)\in {\mathcal{G}}_{\widetilde{\delta}_n} $, the orbit of $x$ under the circle action $S_{k\alpha_{n+1}}$, say $\mathcal{O}^x$, equidistributed among the elements of ${\mathcal{G}}_{\widetilde{\delta}_n}$ because the sequence  $\{k\alpha_{n+1}\}_{k=0}^{q_{n+1}-1}$ equidistributed along $\T^1.$ \\
%As $\frac{\widetilde{\delta}_n}{2}+ \frac{\delta_n}{n^2}< \frac{2\delta_n}{n^2},$ and with our specific function $P_n,$ implies $P_n({\mathcal{G}}_{\widetilde{\delta}_n}) \subset \mathcal{G}_{\delta_n},$ therefore the whole orbit $P_n\mathcal{O}^x$ is trapped inside the elements $G_{i,j}^{n}$ of $\mathcal{G}_{\delta_n}.$ 
\remark\label{re:5.3b}For $x\in \T^1\times(2\e_n^{(2)},\e_n^{(3)})$, since the sequence $\{{k\alpha_{n+1}}\}$  equidistributed over $\T^1$, the orbit of $x$ (say $\mathcal{O}^x$) under the $S_{\alpha_{n+1}}$ equidistributed among the element of $\mathcal{G}_{n}$. There are at most $(4\e_n^{(2)}\frac{q_{n+1}}{q_n})$ exceptional points that are trapped inside the error region $\E_n^g$ (see remark \ref{re:3.5}). Therefore, any element $G_{i,j}^n\in \mathcal{G}_n$ captures at least $\left(1-4\e_n^{(2)}\right)\frac{q_{n+1}}{s_nq_n}$ points of the orbit $\mathcal{O}^x$. 
\remark \label{re:5.3a} Note that under the transformation $\phi_n,$ the elements of $\mathcal{G}_{n}$ map to the elements of $\mathcal{Y}_n.$ In particular,  $\phi_n^g(G_{i,j}^n)= Y_{i,s_n-j}^n$ and conversely, $\phi_n^g(Y_{i,j}^n)= G_{i,s_n-j}^n$ for all $i,j$. By construction, the maps $\phi_{n}^w,\phi_{n}^m$ and $P_n$ act as an identity on the set $\mathcal{G}_{n}$.
%In fact, the number of iterates $k$ such that $P_n\mathcal{O}^x$ lies inside these domains is at least $ \left(1- \frac{4}{q_n^{2/3}}\right)\frac{q_{n+1}}{q_ns_n} \geq \left(1- \frac{4}{n^2}\right)\frac{q_{n+1}}{q_ns_n}$ 
%\left(1- \frac{4\delta_n}{n^2}\right).\frac{q_{n+1}}{q_ns_n}.
%For $\frac{\delta_n}{2}\leq \epsilon_n$ ${\rr{Check \e_3 or \e_n !}}$ and with specific tilding function $P_n,$ the orbit $P_n\mathcal{O}^x$ is almost trapped inside the elements $G_{i,j}^n$ of $\mathcal{G}_n.$ In fact, the number of iterates $k$ such that $P_n\mathcal{O}^x$ lies these domains is at least $\left(1- \frac{n\delta_n}{2}\right).\frac{q_{n+1}}{q_n^2}.$
\begin{prop}\label{pr:1a}
For $\epsilon>0$, consider  $(\frac{\sqrt{2}}{q_{n}}, \epsilon)$-uniformly continuous function $\psi : \mathbb{T}^2 \longrightarrow \mathbb{R}$, i.e. $\psi(B_\frac{\sqrt{2}}{q_{n}}{(x)})\subset B_{\epsilon}(\psi(x))$. The point
 $x \in \T^1\times (2\e_n^{(2)},\e_n^{(3)})$ satisfy the following estimate:
\begin{align}
\left|\frac{1}{q_{n+1}} \sum_{i=0}^{q_{n+1}-1} \psi(\phi_n\circ P_n \circ S_{\alpha_{n+1}}^{i}(x)) -   \int_{\T^2} \psi d\mu \right| \leq 4\epsilon + \frac{2}{nr}\|\psi\|_0 
\end{align}
\end{prop}

\begin{proof}
Fix $x\in \T^1 \times (2\e_n^{(2)},\e_n^{(3)}) $. Since the orbit of $x$ under the $S_{\alpha_{n+1}}$ is being almost trapped inside the elements of $\mathcal{G}_{n},$ therefore there exist a $i_0\in \mathbb{N}$ such that $S_{\alpha_{n+1}}^{i_0}(x) \in G_{i,s_n-j}^n$ for some $i,j \in \mathbb{N}.$  Under the action of $\phi_n$ and by Remark (\ref{re:5.3a}) and (\ref{eq:3a}), we have $$\phi_n\circ P_n\circ S_{\alpha_{n+1}}^{i_o}(x) \in Y_{i,j}^n \subset \Delta_{i,j}^n$$
%Consider  $z= S^{i_0}(x)$ where $z  \in \Gamma_{i',\epsilon_n',k_nq_n} $ for some $i',i_0\in \mathbb{N}$. With the action of $h_{n+1}$, $h_{n+1}(z) \in \left [\frac{i}{q_n},\frac{i}{q_n}+2\epsilon_n'\right) \times \left[\frac{j}{2k_n},\frac{j+1}{2k_n} \right) \subset \Delta_{i,j}^t$ for some $i, j\in \mathbb{N}$ and $\Delta_{i,j}^t\subset R_t$ for $t=0,1.$ 
Therefore for any $y\in \Delta_{i,j}^n$, we have 
 $$d(\phi_n\circ P_n \circ S_{\alpha_{n+1}}^{i_o}(x),y )\leq diam(\phi_n\circ P_n \circ S_{\alpha_{n+1}}^{i_o}(x), y)\leq \sqrt{2}/q_n.$$
 Using the hypothesis on $\psi$, we have 
 $|\psi(\phi_n\circ P_n \circ S_{\alpha_{n+1}}^{i_o}(x))- \psi(y)|< 2\epsilon.$ 
 Take the average for all $y \in \Delta_{i,j}^n $ in the above equation, we get
 $$|\psi(\phi_n\circ P_n \circ S_{\alpha_{n+1}}^{i_o}(x) )-\frac{1}{\mu(\Delta_{i,j}^n)} \int_{\Delta_{i,j}^n} \psi(y) d\mu|< 2\epsilon$$
Let's denote $J_{\Delta} = \{ k\in {0,1,\ldots,q_{n+1}-1}\ : \phi_n\circ P_n \circ S_{\alpha_{n+1}}^{k}(x)  \in \Delta \}$ for all $\Delta\in \mathcal{D}_n.$ %Note that the orbit $\{S_{\alpha_{n+1}}^k(x)\}_{k=0}^{\alpha_{n+1}-1}$ is been almost trapped by the elements $G_{i,j}^n$ except the points  and $\phi_n(G_{i,j}^n)\subset\Delta_{i,j}^n$, 
By remark(\ref{re:5.3b}), we have $|J_{\Delta}|>\left(1- \frac{2}{nr}\right)\frac{q_{n+1}}{s_nq_n}$  (use $4\e_n^{(2)}<\frac{2}{nr}$). Now using the count on $|J_{\Delta}|$ and triangle inequality in the above equation, we get
%{ Note that the orbit of x under the action of ${h_{n+1}\circ S_{\alpha_{n+1}}^k}$ is equidistributed along the$\left [\frac{i}{q_n},\frac{i}{q_n}+\frac{2\epsilon_n'}{q_n}\right) \times \mathbb{T}^1$ for all $i= 0,...,q_{n+1}-1$. Therefore, the number of iterates $k\in \{0,..,q_{n+1}-1\}$ such that $h_{n+1}\circ S_{\alpha_{n+1}}^k(x)$ is contained in $\Delta $ is at least $\frac{q_{n+1}}{2k_nq_n}.$ Using the count on  $|J_{\Delta}|$ }
 %$$|\frac{1}{q_{n+1}} \sum_{i \in J_{\Delta}} \rho(h_{n+1}S_{1/q_{n+1}}^i x)) - \frac{1}{q_{n+1}} \sum_{i \in J_{\Delta}}\frac{1}{\mu(\Delta_{i,j}^t)}\int_{\Delta_{i,j}^t} \rho d\mu| < \frac{|J_{\Delta}|2\epsilon}{q_{n+1}}$$
%With  $\mu(\Delta_{i,j}^t)= \frac{1}{2k_nq_n}$, we obtain
\begin{align}
\left|\frac{1}{q_{n+1}} \sum_{i \in J_{\Delta}} \psi(\phi_n\circ P_n \circ S_{\alpha_{n+1}}^{i}(x) ) - \int_{\Delta_{i,j}^n} \psi d\mu\right| &< 2\epsilon.\mu(\Delta_{i,j}^n) + \frac{2}{nr}(\|\psi\|_0 + 2\epsilon)\mu(\Delta_{i,j}^n) \nonumber\\
&< \left(4\epsilon + \frac{2}{nr}\|\psi\|_0\right)\mu(\Delta_{i,j}^n)
\end{align}
%The latter using the fact $n\delta_n < \frac{2}{n^2}.$ 
Since the last inequality holds for arbitrary $\Delta \in \mathcal{D}_n,$ therefore, we conclude
\begin{align}
\left|\frac{1}{q_{n+1}} \sum_{i=0}^{q_{n+1}-1} \psi(\phi_n\circ P_n \circ S_{\alpha_{n+1}}^{i}(x)) -   \int_{\T^2} \psi d\mu \right| &\leq 
 \left| \sum\limits_{\Delta\in \mathcal{D}_n}\left( \frac{1}{q_{n+1}} \sum_{i \in J_{\Delta}}  \psi(\phi_n\circ P_n \circ S_{\alpha_{n+1}}^{i}(x) - \int_{\Delta} \rho d\mu \right)\right|  \nonumber
\\
 & \ \ \ + \frac{q_n}{q_{n+1}}||\psi||_0   \nonumber\\
& \leq 4\epsilon + \frac{2}{nr}\|\psi\|_0\nonumber
\end{align}
\end{proof}
%\remark In our explicit construction, the measure constructed on the support of $\mathcal{G}_{\widetilde{\delta}_n},$ is generic measure on $\T^2$ and its a non-ergodic measure on $\mathbb{T}^2.$\\

\noindent\textbf{Proof of Theorem A:}
%In section, we introduced This time there are different parts of the torus T2 introduced with distinct aims. On the one hand, we will divide it into r sets Nt by requirements on the x2- coordinate. Each set naturally supports an absolutely continuous probability measure μt given by the normalized restriction of the Lebesgue measure μ. These will enable us to build the ergodic invariant measuresThese parts are measure theoretically insignificant because the measure of these sets will convergeto zero as
%We  construct a map $T\in \textit{Diff}^{\infty}(\mathbb{T}^2,\mu)$ with a non-ergodic measures, written as the sum of two ergodic invariant measure and has generic points.
We will construct a minimal map $T \in \text{Diff}^{\infty}(\mathbb{T}^2,\mu)$, obtained by (\ref{eq:3a}),(\ref{eq:3b}), and (\ref{eq:3c}) for any Liouville $\alpha$ satisfying (\ref{eqn:5.5}), has distinct $r$ weak mixing measures $\mu_t$ and have the Lebesgue measure $\mu$ as a generic measure. Let's fix a countable set of Lipshitz functions  
 $\Psi = \{ \psi_i\}_{i\in \mathbb{N}}$, which is dense in $C^0(\mathbb{T}^2,\mathbb{R}).$ 
 Denote $L_n$ as a uniform Lipshitz constant for $\psi_1,\psi_2,...,\psi_n.$ Choose $q_{n+1}= l_{n}k_nq_n^2$ large enough by choosing $l_n$ arbitrarily large enough such that it satisfies:
 %$$q_{n+1}> n^2 max\limits_{0\leq i\leq n} L_n. \sqrt{2}$$
 \begin{align}
l_n> n^2. \|DH_{n-1}\|_{n-1}.\|Dg_n\|_0 max_{0\leq i\leq n} L_n. \label{eq:l_n}
 \end{align}
 This assumption implies that $\psi_1H_{n-1}g_n,\psi_2H_{n-1}g_n,...,\psi_nH_{n-1}g_n$ are
 $(\frac{\sqrt{2}}{q_n}, \frac{2}{nr})-$ uniformly continuous. 

{\it{ claim 1: The point $x=(0,\frac{\e_n^{(3)}-2\e_n^{(2)}}{2})$ is a generic point for the Lebesgue measure $\mu$ on the $\T^2.$}}\\
Using the fact $h_{n}$ is measure preserving and acts as an identity on the boundary of the unit square, precisely $h_{n}(x)= x$ for all $n$, and $g_n$'s acts as horizontal translation on $\T^2$, we get $H_n^{-1}(x)= x'\in \T^1 \times (2\e_n^{(2)},\e_n^{(3)})$. Now applying the proposition \ref{pr:1a} with $\epsilon = \frac{2}{nr},$ $1\leq k\leq n$, and for $x'\in \mathbb{T}^1\times (2\e_n^{(2)},\e_n^{(3)}) $, we get
\begin{align}
    \left|\frac{1}{q_{n+1}}\sum_{i=0}^{q_{n+1}-1} \psi_k ({H}_{n}S_{{\alpha}_{n+1}}^i x') - \int_{\mathbb{T}^2} \psi_k {H_n} d\mu \right| < \frac{2}{nr} \| \psi_k\|_0  + \frac{8}{nr}.\label{eq:5b}
\end{align}
Using relation (\ref{eq:3b}) and the convergence estimate (\ref{eq:4a}), implies that  for every $\psi_k\in \Psi:$
$$\left|\frac{1}{q_{n+1}}\sum_{i=0}^{q_{n+1}-1} \psi_k(T^i x) - \int_{\mathbb{T}^2} \psi_k d\mu \right| <  \frac{2}{nr} \|\psi_k\|_0  + \frac{8}{nr}+ \frac{1}{2^{n+1}}. $$
Using the triangle inequality, we obtain the claim as $x$ is a generic point for $\mu$.
 $$\lim_{N\longrightarrow \infty}\frac{1}{N}
 \sum_{i=0}^{N-1} \psi_k(T^ix)\longrightarrow \int_{\mathbb{T}^2} \psi_k d\mu. $$
 
 In order to prove the map $T$ is weak mixing w.r.t. to an invariant measure $\mu_t$, we will apply proposition \ref{eq:pr1} on each set $N^t,(t=0,\ldots, r-1)$ which supports $\mu_t$ (see ( \ref{eq:3.1b})). For that consider the sequence $(\mathfrak{m}_n)$ and decomposition $\eta_n^t$  described in section $(\ref{sec:5.1a}-\ref{sec:5.1b})$, and it is enough to show that $\eta_n^t\rightarrow \varepsilon$ and 
 the diffeomorphism $\Phi_n(I_n) = \phi_n\circ P_n\circ S_{\alpha_{n+1}}^{\mathfrak{m}_n} \circ P_n^{-1}\circ \phi_n^{-1}(I_n)$ is $(0,2/3q_n,0)$ -distributes for any $I_n \in \eta_n.$ 
%Consider a sequence $(m_n)$ described in  section$(\ref{sec:5.1a})$ and it is enough to prove that  (\ref{sec:5.1a})
%$\eta_n\longrightarrow 0$ and the diffeomorphism $\Phi_n(S_n) = \phi_n\circ S_{\alpha_{n+1}}^{m_n}\circ \phi_n^{-1}(S_n)$ is $(0,2/3q_n,0)$ -distributes for any $I_n \in \eta_n.$ 
Clearly, $\eta_n\rightarrow \epsilon$, since $\eta_n$ consists of all intervals of each length less than $1/q_n.$
%{\it{Claim - $\Phi(I_n)$ is $(0,2/3n,0)$ -distributed for $I_n \in \eta_n.$}}\\
By lemma (\ref{le:5.1b}), for any $I_n\in \eta_n^t$, $J = \pi_{y}(\Phi_n(I_n)) = \left[ \frac{t}{r}+\frac{2}{3nr},\frac{t+1}{r}-\frac{2}{3nr}\right]$ and $\Phi_n(I_n)$ is a vertical interval. Hence we take $\delta = 2/3n$ and $\gamma = 0$. Finally, the restriction of $\Phi_n(I_n)$ being an affine map, verify the condition for $\epsilon = 0.$ Therefore the map $T$ is a weak mixing w.r.t to the measure $\mu_t (t=0,\ldots, r-1)$. One can ref. to \cite{FS} for more detailed proof. \\
The map $T$ is minimal and has been proved in proposition \ref{pr:2a}, and this completes the proof.
\remark The measure $\mu =\mu_0 +\mu_1+\ldots+\mu_{r-1}$ is a nonergodic Lebesgue measure but a generic measure on the $\T^2$.
\section{Construction of the Generic sets}
In order to prove theorem C and theorem D, we construct a $T\in \text{Diff}^{\infty}(\T^2, \mu)$ using the Approximation by conjugation scheme as done in the last section but will  modify the combinatorics in the above setup to get the desired result. First, we define the combinatorics such that the set $\mathrm{B}\supseteq \{0\} \times C$, where $C$ is the middle third Cantor set, consists of all the generic points of the system and the set $\mathrm{NB}\supseteq \{0\} \times C^c$ , where $C^c= [0,1]\backslash C$, contains all the non-generic points.
% At the n stage of the scheme, we consider $\T^2=\mathrm{G}_n\cup \mathrm{NG}_n$ from section(\ref{sec:6a}), where the set $\mathrm{G}_n$ contains all the generic points and the set $\mathrm{NG}_n$ contains all the non-generic points of the system. 

\subsection{Explicit set-up}\label{sec:6a}
% First we introduce the following partition of $\T^2$ as follows: 
%\begin{align}
%\mathcal{G}_{n}  &:=\left\{{G}_{i,j}^n :=\left [\frac{i}{q_n} +\frac{j}{q_n^2} ,\frac{i}{q_n}+\frac{j+1}{q_n^2}\right) \times \left[\frac{-\delta_n}{2} , \frac{+\delta_n}{2} \right) \ \ : 0\leq i< q_n, 0\leq j< q_n \right\} \nonumber \\
%\tau_{n}  &:=\left\{{\Delta}_{i,j}^n :=\left [\frac{i}{q_n}  ,\frac{i+1}{q_n}\right) \times \left[\frac{j}{3^n}, \frac{j+1}{3^n} \right) \ \ : 0\leq i< q_n, 0\leq j< 3^n \right\} 
%\end{align}
Consider the following collection of disjoint subsets of $\T^2:$ $\T^2 = (\mathrm{G} \cup \mathrm{NG})$  such that
 %The following results allows us to prove genericity of a measure which is supported on the region 
 \begin{align}
     \mathrm{G}&= \bigcap_{n\geq 1} \mathrm{G}_n = \T^1 \times C ,\ \  \text{where} \ \mathrm{G}_{n} = \T^1 \times \bigcup_{l=0}^{2^{n}-1} I_{l}^n , \label{eq:6a} \ \ \\
     \mathrm{NG} &= \bigcup_{n\geq 1} \mathrm{NG}_n = \T^1 \times ([0,1]\backslash C), \ \ \text{where}
     \ \ \mathrm{NG}_{n} = \T^1 \times\bigcup_{k=0}^{n-1} \bigcup_{l=0}^{2^{k-1}-1} J_{l}^k , \label{eq:6.1.b} \ \  
    \end{align}
where $I_{l}^n$ and $J_l^n$ are intervals of $[0,1]$ as defined in section \ref{sec:2.2a}. We split the interval $J_0^1$ into two halves as $J_0^1= \hat{J}_{0}^1\cup \hat{J}_{1}^1$, where $\hat{J}_{0}^1= \left(\frac{1}{3},\frac{1}{2}\right)$ and $\hat{J}_{1}^1= \left(\frac{1}{2},\frac{2}{3}\right)$.\\
Additionally, we introduce the following partition of $\T^2$ for any
natural number sequence $q_n$ and $s_n> q_n$ as follows: 
\begin{align}
\mathrm{G}_{n} &:=\left\{\mathcal{I}_{i_1,i_2}^n=\left[\frac{i_1}{s_nq_n},\frac{i_1+1}{s_nq_n}\right) \times I_{i_2}^n \ \ : 0\leq i_1< s_nq_n, 0\leq i_2< 2^n-1  \right\},  \\
%\mathrm{NG}_{n} &:=\bigg\{\mathcal{J}_{i_1,i_2}^{n,k}=\left[\frac{i_1}{s_nq_n},\frac{i_1+1}{s_nq_n}\right) \times J_{i_2}^k; \ \ : 2\leq k\leq n, \  0< i_2< 2^{n-1}-1 \nonumber \\&\qquad \mathcal{J}_{i_1,i_2}^{n,1}=\left[\frac{i_1}{s_nq_n},\frac{i_1+1}{s_nq_n}\right) \times \hat{J}_{i_2'}^1 \ \ 0\leq i_1< s_nq_n, i_2'= 0,1 \bigg\},  \\
{\mathrm{NG}}_{n}  &:=
\left.\begin{cases}
  \mathcal{J}_{i_1,i_2}^{n,k}=\left[\frac{i_1}{s_nq_n},\frac{i_1+1}{s_nq_n}\right) \times J_{i_2}^k; \ \   \ 2\leq k\leq n, \  0< i_2< 2^{n-1}-1, \\
\mathcal{J}_{i_1,i_2'}^{n,1}=\left[\frac{i_1}{s_nq_n},\frac{i_1+1}{s_nq_n}\right) \times \hat{J}_{i_2'}^1 \ \  %0\leq i_1< s_nq_n
: \ 0\leq i_1< s_nq_n,  \  i_2'= 0,1 
\end{cases}
\right\},\\
\mathrm{V}_{n}  &:=\bigg\{{\mathcal{V}}_{i_1,i_2,i_3}^n =\left [\frac{i_1}{q_n} +\frac{i_2}{3^nq_n} ,\frac{i_1}{q_n}+\frac{i_2+1}{3^nq_n}\right) \times \left[\frac{i_3}{s_n} , \frac{i_3+1}{s_n} \right) \ \ : 0\leq i_1< q_n,  \nonumber \\ &\qquad\qquad\qquad\qquad\qquad\qquad\qquad\qquad \qquad \qquad \qquad \qquad 0\leq i_2< 2^n-1,\  0\leq i_3< s_n\bigg\}, \label{eq:6h} \\
\mathrm{W}_{n}  &:=\left.
\begin{cases}
  \mathcal{W}_{i_1,i_2}^{n,k} =\left [\frac{i_1}{q_n} +\frac{2^k}{3^kq_n} ,\frac{i_1}{q_n}+\frac{2^k}{3^kq_n}+\frac{2^{k-1}}{3^{k}q_n}\right) \times \left[\frac{i_2}{s_n2^{k-1}} , \frac{i_2+1}{s_n2^{k-1}} \right); \  2\leq k \leq n; \\
 \mathcal{W}_{i_1,i_2'}^{n,1} =\left [\frac{i_1}{q_n} +\frac{2}{3q_n} ,\frac{i_1+1}{q_n}\right) \times \left[\frac{i_2'}{2s_n} , \frac{i_2'+1}{2s_n} \right)   : 0\leq i_1< q_n, \ 0\leq i_2< s_n,  i_2'=0, 1
 \end{cases}\right\}.
 %{\mathrm{D}}_{n}  &:=\left\{{\Delta}_{i_1,i_2}^n :=\left [\frac{i_1}{q_n}  ,\frac{i_1+1}{q_n}\right) \times \left[\frac{i_2}{s_n}, \frac{i_2+1}{s_n} \right) \ \ : 0\leq i< q_n, 0\leq j< s_n \right\}\\
\end{align}

\subsubsection{The Conjugation map \texorpdfstring{$\overline{\phi}_n$}{Lg}}
Now we define the following permutation maps $\widetilde{\phi}_n:\T^2\longrightarrow\T^2$ of the above partition $\mathrm{G}_n\cup\mathrm{NG}_n$ which maps to the elements of partition $\mathrm{V}_n\cup\mathrm{W}_n.$ Consider the map $\widetilde{\phi}_{n}: \left[\frac{i}{q_n},\frac{i+1}{q_n}\right) \times \T^1 \longrightarrow \left[\frac{i}{q_n},\frac{i+1}{q_n}\right) \times \T^1$ as following and extend it to the whole $\T^2$ as $\frac{1}{q_n}$-equivariantly.
\begin{align}
\widetilde{\phi}_n(\mathcal{I}_{i_1,i_2}^n)= \mathcal{V}_{j_1,j_2,j_3}^n \ \ \  &\text{where} \ \ \ j_1=\left\lfloor{\frac{i_1}{s_n}}\right\rfloor,   \ j_2=i_2,\ j_3= i_1\mod s_n, \\
\widetilde{\phi}_n(\mathcal{J}_{i_1',i_2'}^{n,k})= \mathcal{W}_{j_1',j_2'}^{n,k}\ \ \  &\text{where}\ \ \ j_1'= \left\lfloor{\frac{i_1'}{s_n}}\right\rfloor, 
%j_2=i_2'.s_n+ i_1'\mod s_n, \forall  i_2'\neq 0\\
%\widetilde{\phi}_n(\mathcal{J}_{i_1',i_2'}^n)= \mathcal{W}_{j_1',j_2'}\ \ \  &\textit{where}\ \ \ j_1'= \left\lfloor{\frac{i_1'}{s_n}}\right\rfloor, j_2'= i_1'\mod s_n , i_2'=a\\
%\widetilde{\phi}_n(\mathcal{J}_{i_1',i_2'}^n)= \mathcal{W}_{j_1',j_2'}\ \ \  &\textit{where}\ \ \ j_1'= \left\lfloor{\frac{i_1'}{s_n}}\right\rfloor, j_2'= s_n + i_1'\mod s_n , i_2'=b
j_2' = 
    \begin{cases} 
      i_2'.s_n+i_1'\mod{s_n} \  &{\text{for}} \  2\leq k\leq n\\
       i_1'\mod{s_n}\ &{\text{for}}\ k=1 \ \& \ i_2'=0\\
       s_n + i_1'\mod{s_n}\  &{\text{for}}\ k=1 \ \& \ i_2'=1
   \end{cases} 
\end{align}
\begin{figure}[ht]\label{fig:02}
    \centering
    \includegraphics[width=.7\textwidth]{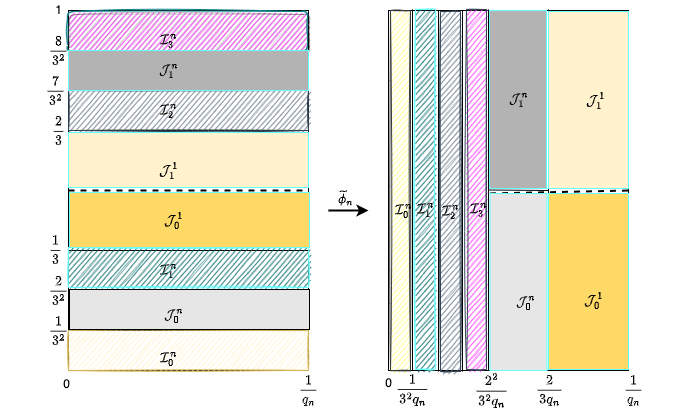}
    \caption{An example of action $\widetilde{\phi}_n$ on the elements of $\mathrm{G}_n\cup\mathrm{NG}_n$ for $n=2.$ }
\end{figure}
Indeed, the map $\widetilde{\phi}_n$  is a measure preserving map on the $\T^2$ and can be better understood by following rectangles as  
\begin{align}
    &\widetilde{\phi}_n\left(\left[ \frac{i}{q_n},\frac{i+1}{q_n}\right) \times I_l^n \right) = \left[ \frac{i}{q_n} + \frac{l}{3^nq_n},\frac{i}{q_n}+ \frac{l+1}{3^nq_n}\right) \times \T^1  \ \   \label{eqn:6a}\\
   &\widetilde{\phi}_n\left(\left[ \frac{i}{q_n},\frac{i+1}{q_n}\right) \times J_l^k \right) = \left[ \frac{i}{q_n} + \frac{2^k}{3^kq_n},\frac{i}{q_n}+ \frac{2^k}{3^kq_n}+ \frac{2^{k-1}}{3^{k}q_n}\right) \times \left(\frac{l}{2^{k-1}} , \frac{l+1}{2^{k-1}}\right) \ ; \  2\leq k\leq n \label{eq:6b}\\
    %\ \ l=0,1,..,2^{n-1}-1 , n \geq 2\nonumber \\
%\end{align}
%\begin{align}
  &\widetilde{\phi}_n\left(\left[ \frac{i}{q_n},\frac{i+1}{q_n}\right) \times \left(\frac{1}{3},\frac{1}{2}\right) \right) = \left[ \frac{i}{q_n} + \frac{2}{3q_n},\frac{i}{q_n}+ \frac{2}{3q_n}+ \frac{1}{3q_n}\right) \times \left(0 , \frac{1}{2}\right) \label{eq:6c} \\
  &\widetilde{\phi}_n\left(\left[ \frac{i}{q_n},\frac{i+1}{q_n}\right) \times \left(\frac{1}{2},\frac{2}{3}\right) \right) = \left[ \frac{i}{q_n} + \frac{2}{3q_n},\frac{i}{q_n}+ \frac{2}{3q_n}+ \frac{1}{3q_n}\right) \times \left(\frac{1}{2} , 1 \right)\label{eq:6d} 
\end{align}

\remark Observe that, in (\ref{eqn:6a}), 
$\widetilde{\phi}_n$ takes
 very thin horizontal strip $\mathcal{I}_{l}^n= \T^1\times I_l^n$ and distributes it in the vertical direction all over the torus periodically, which will allow us to obtain generic points whose orbits are uniformly distributed all over the torus. Also, note that the measure of such a set ,containing generic points, is zero.
Whereas in (\ref{eq:6b}), (\ref{eq:6c}) and (\ref{eq:6d}), $\widetilde{\phi}_n$ take $\mathcal{J}_l^k= \T^1\times J_l^k$ and distributes it such that it remain within the  region $\left(\frac{l}{2^{k-1}} , \frac{l+1}{2^{k-1}}\right)$, which produces the non-generic points, see Figure {\ref{fig:02}}.

We can extend this map to a smooth map $\widetilde{\phi}_n:\T^2\longrightarrow\T^2$ as $\frac{1}{q_n}$ equivariantly. Using the fact that any permutation map defined on the torus can be well approximated by a smooth map that preserves the same combinatorics of the permutation inside the torus and acts as an identity on the boundary of $\T^2$.
This assertion builds upon the lemma ({\ref{lem:01}}) that there is $C^{\infty}$ measure-preserving map that rotates the disc of radius $R-\delta$ inside $[0,1]\times[0,1]$ by an angle $\pi$ and which is identically equal to zero in an arbitrarily small neighbourhood
of the disc of radius $R$, and acts as an identity on the boundary of $[0,1]\times[0,1]$. Hence
any permutation $\sigma$ can be written as a composition of transposition(rotation). Therefore the smooth maps can closely approximate each transposition by choosing a small enough $\delta$ in the above lemma. The analogous result has been used in \cite{FW1}, \cite{FS} and \cite{FSW}.
Let's denote $\overline{\phi}_{n}$ to be the smooth diffeomorphism  obtained by the permutation map $\widetilde{\phi}_{n}$ on $\mathbb{T}^2$. %It preserves the same action on the elements of  $\mathcal{G}_{k_nq_n}$ with an error set $E_{n}$ of small measure and acts as an identity on the boundary of $\mathbb{T}^2$ ,and commutes with $S_{\alpha_n}$
%$${h}_{n+1}\circ S_{\alpha_{n}}= S_{\alpha_{n}}\circ {h}_{n+1}$$
%all the other rectangles are mapped torectangle of small diameter but they remains within the region Rt on the torus. These Region Rtwill become the support of the invariant measure
\subsubsection{The conjugation map \texorpdfstring{$h_n$}{Lg}}
Here we define our final conjugation diffeomorphism as 
\begin{align}
    h_n= \overline{\phi}_n\circ P_n,\label{eq:6i}
\end{align}
where $\overline{\phi}_n$ is the smooth approximation of the map $\widetilde{\phi}_n$ and the diffeomorphism  $P_n$ from section{\ref{sec:3a}} with the smooth map $\kappa_n: \T^1\longrightarrow [0,1]$. % which is smooth approximation of $\tilde{\kappa_n}$.
In this specific situation, we choose  $\tilde{\kappa_n}:\left[0,\frac{1}{s_nq_n}\right]\longrightarrow\T^1$ defined as \begin{equation}\label{eqn:1.2}
\tilde{\kappa_{n}}(x) = 
    \begin{cases} 
      \frac{\delta_n2q_ns_n}{n^2}(x) \ \   \  &,x\in [0,\frac{1}{2s_nq_n}) \\
       -\frac{\delta_n2q_ns_n}{n^2}(x)+\frac{2\delta_n}{n^2} \ \   \  &,x\in [\frac{1}{2s_nq_n},\frac{1}{s_nq_n}]
   \end{cases} 
\end{equation}
where $\delta_n= \frac{1}{e^{3^n}}$. Now, extend this map $\tilde{\kappa_n}$ periodically with period $\frac{1}{s_nq_n}$ on $\T^1$ and choose 
$\kappa_n$ to be the smooth approximation of $\tilde{\kappa_n}$ on $\T$ by Weierstrass Approximation Theorem. 
%Here we choose $\kappa_n:\left[0,\frac{1}{q_n}\right]\longrightarrow\T^1$ as $\kappa_n(x)= \frac{1}{n^2}\delta_nq_nx,$ with $\delta_n= \frac{1}{3^n2^{q_ns_n}}$ and extend this map to $\T^1$ with period $\frac{1}{q_n}.$
\remark  The map $P_n$ ensures control of all the orbits, such that no whole orbit of a point is trapped inside the error set, which would guarantee that there are no other generic points w.r.t. to $\mu$ measure outside the set $\mathrm{B}$ and no other non-generic points outside the set $\mathrm{NB}$. But, this is not the case in theorem A, where we don't care about the number of generic points.
\remark Note that $h_n\circ S_{\alpha_n}= S_{\alpha_n} \circ h_n,$   since both the maps $\overline{\phi}_n$ and $P_n$ commute with $S_{\alpha_n}$ by construction.
\subsubsection{Convergence and Estimates}
To exclude the region where we don't have control over the combinatorics, we consider a subset $E_n$ of $\T^1$ as  
\begin{align}
%E_n= \bigcup_{l=0}^{3^n-1}\left[\frac{l}{3^n}-\epsilon_n',\frac{l}{3^n} +\epsilon_n'\right]\\
E_n= \left(\bigcup_{i=0}^{s_nq_n-1}\left[\frac{i}{s_nq_n}-\frac{\epsilon_n'}{2},\frac{i}{s_nq_n}+\frac{\epsilon_n'}{2}\right] \times \T^1\right) \bigcup \left(\bigcup_{l=0}^{3^n-1} \T^1 \times \left[\frac{l}{3^n}-\frac{\epsilon_n'}{2}, \frac{l}{3^n}+\frac{\epsilon_n'}{2}\right]\right), 
\end{align}
where $\epsilon_n'$ is chosen such that $\mu(E_n)< \frac{1}{e^{3^n}}.$ Denote the set $F_n= \T^2\backslash E_n$ such that $\mu(F_n)> 1-\frac{1}{e^{3n}}.$\\
Hereby we introduce the following collection of sets that corresponds  to ``trapping generic zones" and ``trapping nongeneric zones" respectively  (for $i_1=0,1,\ldots,q_ns_n-1$),
\begin{align}
    \mathcal{X}_{i_1,t_1}^n &= P_n^{-1}\left(\mathcal{I}_{i_1,t_1}^n\bigcap F_n  \right), \ \  \ \  t_1= 0,1,\ldots,2^{n}-1 \label{eq:6f}\\
    \mathcal{Y}_{i_1,t_2}^{n,k} &= P_n^{-1}\left(\mathcal{J}_{i_1,t_2}^{n,k} \bigcap F_n  \right),\ \ \ \  t_2= 0,1,\ldots,2^{n-1}-1, \ \ 1\leq k\leq n. \label{eq:6g}
\end{align}
%\begin{align} X_{i_1,t_1}^n &= P_n^{-1}\left(\left(\left[\frac{i_1}{s_nq_n},\frac{i_1+1}{s_nq_n}\right)\times I_{t_1}^n\right) \bigcap F_n  \right), \ \  \ \  t_1= 0,1,..,2^{n}-1\\
   % Y_{i_1,t_2}^n &= P_n^{-1}\left(\left(\left[\frac{i_1}{s_nq_n},\frac{i_1+1}{s_nq_n}\right)\times J_{t_2}^n\right) \bigcap F_n  \right),\ \ \ \  t_2= 0,1,..,2^{n-1}-1
%\end{align}
\begin{lemma}
For any $x\in \T^1\times I_{t_1}^n,$ for $t_1=0,1,\ldots,2^{n}-1,$ the orbit  $\{S_{\alpha_{n+1}}^k(x)\}_{k=0}^{q_{n+1}-1}$ 
meets every set $\mathcal{X}_{i_1,t_1}^n,$ for any $ i_1= 0,1,\ldots,s_nq_n-1.$ Moreover, the number of
iterates of orbit lie in every set $\mathcal{X}_{i_1,t_1}^n$ is  at least
$\left(1- \frac{2}{n^2}\right)\frac{q_{n+1}}{3^ns_nq_n}.$ 
\end{lemma}
\begin{proof}
Fix any $x\in \T^1\times I_{t_1}^n,$ 
 the orbit of $x$ under the circle action $S_{\alpha_{n+1}}^k$, say $\mathcal{O}^x$, is  equidistributed along $\T^1\times I_{t_1}^n$ because the sequence  $\{k\alpha_{n+1}\}_{k=0}^{q_{n+1}-1}$ is equidistributed along $\T^1.$ In particular, $\mathcal{O}^x$ is equidistributed along the elements $ \mathcal{I}_{i_1,t_1}^n= \left[\frac{i_1}{s_nq_n},\frac{i_1+1}{s_nq_n}\right) \times I_{t_1}^n$ for every $i_1=0,1,\ldots,s_nq_n-1.$ Note that 
 $$\left[\frac{i_1}{s_nq_n}+\frac{\epsilon_n'}{2},\frac{i_1+1}{s_nq_n}-\frac{\epsilon_n'}{2}\right] \times \left[\frac{t_1}{3^n}+\frac{\epsilon_n'}{2}, \frac{t_1+1}{3^n}-\frac{\epsilon_n'}{2}\right]\subset \mathcal{I}_{i_1,t_1}^n\bigcap F_n.$$
 The map $P_n$ acts as vertical translation on $\T^2,$ and with the choice of $\kappa_n$ function, the net translation caused by the section $\left[\frac{i_1}{s_nq_n}+\frac{\epsilon_n'}{2},\frac{i_1+1}{s_nq_n}-\frac{\epsilon_n'}{2}\right]$ inside the section $\left[\frac{t_1}{3^n}+\frac{\epsilon_n'}{2}, \frac{t_1+1}{3^n}-\frac{\epsilon_n'}{2}\right]$ is almost $\frac{\delta_n}{n^2s_n}$. Due to  $\frac{\delta_n}{n^2} <\frac{1}{n^23^n}$, we can estimate
% the set $A_{i_1,i_2}^n$ under the map $P_n^{{-1}},$ satisfy $\pi_2(P_n^{-1}(A_{i_1,i_2}^n))=\T^1.$ Due to $\frac{\delta_n}{n^2} <\frac{1}{3^n}$ and with our precise choice of $P_n,$ we can estimate
\begin{align}
    \mu(\mathcal{X}_{i_1,t_1}^n\cap \mathcal{I}_{i_1,t_1}^n) &\geq (1-2\epsilon_n') \frac{\left|\left[\frac{t_1}{3^n}+\frac{\delta_n}{n^2}, \frac{t_1}{3^n}+\frac{1}{3^n}\right]\right|}{s_nq_n} \nonumber \\
    &\geq (1-2\epsilon_n')\left(1-\frac{3^n\delta_n}{n^2}\right)\frac{1}{3^ns_nq_n} \nonumber\\
    &\geq \left(1-\frac{2.3^n\delta_n}{n^2}\right)\frac{1}{3^ns_nq_n} \geq \left(1-\frac{2}{n^2}\right)\frac{1}{3^ns_nq_n}
\end{align}
Hence, at least $\left(1-\frac{2}{n^2}\right)\frac{q_{n+1}}{3^ns_nq_n}$ number of elements are trapped inside the orbit $\mathcal{O}^x$.
\end{proof}
\remark {\label{rem:6a}}Recall that the image of  $\mathcal{X}_{i_1,i_2}^n$, under the conjugation map $h_{n},$ contained inside $\mathcal{V}_{\lfloor{\frac{i_1}{s_n}}\rfloor, i_2, i_1\mod{s_n}}^n$ and conversely, $\mathcal{V}_{i_1,i_2,i_3}^n$ is uniquely mapped onto 
$\mathcal{X}_{i_1.s_n+i_3,i_2}^n.$ By the above estimate, the number of iterates $ k\in \{0,1,\ldots,q_{n+1}-1\}$ such that 
$h_n \circ S_{\alpha_{n+1}}^{k}(x)  \in \mathcal{V}_{i_1,i_2,i_3}$ for $x\in \T^1\times I_{t_1}^n$ is at least $\left(1-\frac{2}{n^2}\right)\frac{q_{n+1}}{3^ns_nq_n}. $
%(1-\frac{2^{n+1}}{3^n})\frac{q_{n+1}}{s_nq_n}.$ \\
\remark {\label{rem:6b}} Note that under the action of $h_n$, every element from $\mathrm{NG}_n$ transform as  $(\text{for} \ i_2= 0,1,\ldots,2^{n-1}-1)$,
\begin{align}
    h_n\left(\bigcup\limits_{i_1=0}^{s_nq_n-1}\mathcal{Y}_{i_1,i_2}^{n,k}\right)&= \bigcup\limits_{i_1=0}^{s_nq_n-1}\overline{\phi}_n(\mathcal{J}_{i_1,i_2}^{n,k}\cap F_n)\subseteq \T^1 \times \left[\frac{i_2}{2^{k-1}},\frac{i_2+1}{2^{k-1}}\right); \ \ 2\leq k \leq n \\ 
    h_n\left(\bigcup\limits_{i_1=0}^{s_nq_n-1}\mathcal{Y}_{i_1,t}^{n,1}\right)&= \bigcup\limits_{i_1=0}^{s_nq_n-1}\overline{\phi}_n(\mathcal{J}_{i_1,t}^{n,1}\cap F_n) \subseteq \T^1 \times \left[\frac{t}{2},\frac{t+1}{2}\right); \ t=0,1 .
\end{align}
 
\begin{prop}\label{pr:2:a}
For $\epsilon>0$, consider  $(\frac{\sqrt{2}}{q_{n}}, \epsilon)$-uniformly continuous function $\psi : \mathbb{T}^2 \longrightarrow \mathbb{R}$, i.e. $\psi(B_\frac{\sqrt{2}}{q_{n}}{(x)})\subset B_{\epsilon}(\psi(x))$. Then for any
 $x\in \mathrm{G}_n,$ satisfy the following estimate:
\begin{align}
\left|\frac{1}{q_{n+1}} \sum_{i=0}^{q_{n+1}-1} \psi(h_n \circ S_{\alpha_{n+1}}^{i}(x)) -   \int_{\T^2} \psi d\mu \right| \leq 4\epsilon + \frac{2}{n^2}\|\psi\|_0 \label{eq:6e}
\end{align}
\end{prop}
\begin{proof}
For any $x\in \mathrm{G}_n$ and $\Delta_{i_1,i_2}^n \in \Delta_{i,j}^n $(see (\ref{eq:5.3a})). Precisely, $x\in \T^1 \times I_{l}^n$ for some $l$. Since the orbit of $x$ under the $S_{\alpha_{n+1}}^k$ is almost trapped by the domains $\{\mathcal{X}_{t_1,t_2}^n\}$, therefore there exist a $i_0\in \mathbb{N}$ such that $S_{\alpha_{n+1}}^{i_0}(x) \in \mathcal{X}_{i_1.s_n+i_2,l}^n.$ With the action of $h_n$, by (\ref{eqn:1.2}) and remark ({\ref{rem:6a}}), we have $$h_n\circ S_{\alpha_{n+1}}^{i_o}(x) \in \mathcal{V}_{i_1,l,i_2}^n \subset \Delta_{i_1,i_2}^n.$$
Therefore for any $y\in \Delta_{i,j}^n$, we conclude 
 $$d(h_n \circ S_{\alpha_{n+1}}^{i_o}(x),y )\leq diam(h_n \circ S_{\alpha_{n+1}}^{i_o}(x), y)\leq \sqrt{2}/q_n.$$
Now using the hypothesis on $\psi$, we have 
 $|\psi(h_n \circ S_{\alpha_{n+1}}^{i_o}(x))- \psi(y)|< 2\epsilon.$ 
 Take the average for all $y \in \Delta_{i,j}^n $ in the last equation, we get
 $$|\psi(h_n \circ S_{\alpha_{n+1}}^{i_o}(x) )-\frac{1}{\mu(\Delta_{i,j}^n)} \int_{\Delta_{i,j}^n} \psi(y) d\mu|< 2\epsilon.$$
Let's denote $J_{\Delta} = \{ k\in {0,1,\ldots,q_{n+1}-1}\ : h_n \circ S_{\alpha_{n+1}}^{k}(x)  \in \Delta \}$ for all $\Delta\in \mathcal{D}_n,$ where $\mathcal{D}_n$ defined by (\ref{eq:5.3a}).
Using the count estimate described in remark({\ref{rem:6a}}) and triangle inequality in the last equation, we have
\begin{align}
\left|\frac{1}{q_{n+1}} \sum_{i \in J_{\Delta}} \psi(h_n \circ S_{\alpha_{n+1}}^{i}(x) ) - \int_{\Delta_{i,j}^n} \psi d\mu\right| < (4\epsilon+ \frac{2}{n^2}\|\psi\|_0)\mu(\Delta_{i,j}^n) 
\end{align}
Further, we follow the analogous estimation as done in proposition {\ref{pr:1a}}, and we have the estimate(\ref{eq:6e}) as required.
 \end{proof}
\begin{lemma}\label{le:6a}
The sequence of diffeomorphisms  $T_{n} = H_{n} \circ S_{\alpha_{n+1}}\circ H_{n}^{-1},$ such that $H_n=h_1\circ h_2\ldots\circ h_n$ and $h_n$ defined by (\ref{eq:6i}) and $\alpha_{n+1}$ converges to a Liouvillian number, converges to  $T\in \text{Diff}^{\infty}(\mathbb{T}^2, \mu)$ in the $C^{\infty}$ topology. Moreover, for any $m\leq q_{n+1},$ we have 
\begin{equation}\label{eq:6.24}
d_0(T^m,T_{n}^m) \leq \frac{1}{2^{n+1}}, 
\end{equation} 
\end{lemma}
%Following on the line of Windsor and \cite{Ba-Ku} of proposition 3 in  \cite{FSW}. 
\subsubsection{Proof of Theorem C}
 \begin{proof} 
Let's fix a countable set of Lipshitz functions  
 $\Psi = \{ \psi_i\}_{i\in \mathbb{N}},$ which is dense in $C^0(\mathbb{T}^2,\mathbb{R}).$ 
 Denote $L_n$ to be a uniform Lipshitz constant for $\psi_1,\psi_2,\ldots,\psi_n.$ Choose $q_{n+1}= l_{n}k_nq_n^2$ large enough by choosing $l_n$ enuogh arbitrary large such that it satisfies:
 %$$q_{n+1}> n^2 max\limits_{0\leq i\leq n} L_n. \sqrt{2}$$
 \begin{align}
l_n> n^2. ||DH_{n-1}||_{n-1} max_{0\leq i\leq n} L_n. \label{eq:6.1c}
 \end{align}
 The latter assumption guarantees the convergence of sequences of diffeomorphism $\{T_n\}$ and implies that $\psi_1H_{n-1},\psi_2H_{n-1},...,\psi_nH_{n-1}$ are
 $(\frac{\sqrt{2}}{q_n}, \frac{1}{n^2})$-uniformly continuous.\\
{\it{ claim 1: Every point inside the set $\mathrm{B}= \liminf\limits_{n\rightarrow0} \mathrm{B}_n$ is a generic point, where $\mathrm{B}_n= H_n(\mathrm{G})$}}\\
Let $y\in \mathrm{B}$, i.e. $y\in \mathrm{B}_n \ \forall n$ except for finitely many $n.$ Say, $x_n=H_n^{-1}(y)\in \mathrm{G}\subset \mathrm{G}_n.$\\
Apply the propostition {\ref{pr:2a}} with $\epsilon = \frac{1}{n^2},$ $1\leq k\leq n$, and for $x_n\in \mathrm{G}_n$ (see \ref{eq:6a}), 
\begin{align}
\left|\frac{1}{q_{n+1}}\sum_{i=0}^{q_{n+1}-1} \psi_k(H_{n}S_{{\alpha}_{n+1}}^i x_n) - \int_{\mathbb{T}^2} \psi_kH_n d\mu \right| < \frac{2}{n^2} ||\psi_k||_0  + \frac{4}{n^2}.
\end{align}
 
Use the fact $H_{n}$ is area preserving smooth diffeomorphism   and $H_{n}(x_n)= y,$ with the convergence estimate (\ref{eq:6.24}) in the last  equation, which implies for every $\psi_k\in \Psi$
$$\left|\frac{1}{q_{n+1}}\sum_{i=0}^{q_{n+1}-1} \psi_k(T^i y) - \int_{\mathbb{T}^2} \psi_k d\mu \right| <  \frac{2}{n^2}||\psi_k||_0  + \frac{4}{n^2}+ \frac{1}{2^{n+1}}, $$
Using the triangle inequality and we obtain $y$ as a generic point for $\mu$ in the sense of ({\ref{def:1a}}) such that
 $$\lim_{N\longrightarrow \infty}\frac{1}{N}
 \sum_{i=0}^{N-1} \psi_k(T^iy)\longrightarrow \int_{\mathbb{T}^2} \psi_k d\mu. $$
Since $y\in \mathrm{B}$  chosen arbitrarily,, therefore every point $ y\in \mathrm{B}$ is a generic point. %The map $H_n$ being a measure perserving and by relation (\ref{eq:6a}), 
\\
{\it{claim 2: $\text{dim}_H(C)\leq  \text{dim}_H(\mathrm{B}) \leq \text{dim}_H(\mathrm{G})= 1+ \frac{log2}{log3}$}}.\\
By construction, $H_n$ acts as an identity near the boundary of $\T^2$, implying that $\{0\}\times C \subseteq \mathrm{B}_n$ for all $n$. Hence, $\{0\}\times C \subseteq \mathrm{B}$ and $\text{dim}_H(C)\leq  \text{dim}_H(\mathrm{B})$. \\
The right-hand inequality holds by the following inequality: 
$\text{dim}_H(\mathrm{B})\leq \text{dim}_H(\mathrm{B}_n)= \text{dim}_H(\mathrm{G})$
where the first inequality holds true by containment $\mathrm{B}\subseteq \mathrm{B}_n$ and the second equality holds by lemma (\ref{lem:3a}) where $H_n$ being smooth diffeomorphism and $\mathrm{G}$ is a compact set. 
With the product rule of Hausdorff dimension  (\ref{def:1b}), and the fact $\text{dim}_H(C)= \frac{log2}{log3}$ and  (\ref{eq:6a}), we have {$\text{dim}_H(\mathrm{G})= 1+ \frac{log2}{log3}$}.\\
%it satisfy $$\|DH_{n}^{-1}\|_0.d(x,y)\leq d(H_{n}(x),H_{n}(y))\leq \|DH_{n}\|_0.d(x,y) \ \  \forall \ x,y \in \mathrm{G}$$
%Since $H_n$ being a smooth area preserving diffeomorphism on $\T^2$ is bilipschitz and every bilipschitz continuous function preserves the Hauasdroff dimension(``Hausdorff dimensionis a bilipschitz invariant''). Hence, $\dim_H(\mathrm{G}_n)= \dim_H(H_n(\mathrm{G}_n))$ and $\mu(\mathrm{G}_n)= \mu(H_n(\mathrm{G}_n))$.$$\|DH_{n}^{-1}\|_0.d(x,y)\leq d(H_{n}(x),H_{n}(y))\leq \|DH_{n}\|_0.d(x,y) \ \  \forall \ x,y \in G_n$$By relation, $\mathrm{G}_{n+1}\subset \mathrm{G}_{n}$ and $\mathrm{G}= \cap_{n\geq 1} \mathrm{G}_n $
{\it{claim 3: Every point inside the set $\mathrm{NB}=\T^2\backslash \mathrm{B}=\limsup\limits_{n\rightarrow\infty} \mathrm{B}_n^c $ is a non-generic point.}}\\
 With the convergence estimate (\ref{eq:6.24}) and triangle inequality, it is enough to show for $y\in \mathrm{NB}$, 
$$\lim_{N\longrightarrow \infty}\frac{1}{N} \sum_{i=0}^{N-1} \phi(T_n^iy) \not\longrightarrow \int_{\mathbb{T}^2} \phi du \, \ \text{for infinitely many } n \  \text{
and for some} \ \phi\in C^0(\mathbb{T}^2,[0,1]).$$
%Let's assume any $y\in \mathrm{NB}$ is a generic point i.e., $\forall \varepsilon>0, \text{there exist} N_0:$ $$\left|\lim_{N\longrightarrow \infty}\frac{1}{N} \sum_{i=0}^{N-1} \phi(T^iy)- \int_{\mathbb{T}^2}\phi\right| < \varepsilon $$
If $y\in \mathrm{NB}$ then $\forall \ n_0\in \mathbb{N},\  \text{there exist}\  n_1>n_0 : y\in \mathrm{B}_{n_1}^c$, where $\mathrm{B}_{n_1}^c=\T^2\backslash \mathrm{B}_{n_1}$. Say, $x_{n_1}= H_{n_1}^{-1}(y)$. Therefore $x_{n_1}\in \mathrm{NG},$  i.e. ${x_{n_1}}\in \mathcal{J}_l^k \ \text{for some} \ l, k\in \mathbb{N} \ (\text{because} \ \mathrm{NG}=\sqcup_{k}\sqcup_{l} \mathcal{J}_l^k).$
Let's consider $\phi_n = \pi_2\circ H_{n-1}^{-1}$ a continuous function on $\T^2,$ and
by remark (\ref{rem:6b}), we reduced to 
 \begin{align}
     &\phi_{n_1}(T_{n_1}^i(y))= \pi_2\circ h_{n_1}\circ S_{\alpha_{n_1+1}}^i(x_{n_1}) \subset  \left[\frac{l}{2^{k-1}},\frac{l+1}{2^{k-1}}\right) \ \  \forall \ i \in \mathbb{N},\nonumber\\
 \text{i.e.}  \  &\left|\lim_{N\longrightarrow \infty}\frac{1}{N} \sum_{i=0}^{N-1} \phi_{n_1}(T_{n_1}^iy) - \int_{\mathbb{T}^2} \phi_{n_1} d\mu\right| \geq 1/2.\nonumber
 \end{align}
 $$\implies \forall n_0\in \mathbb{N},\  \text{there exist } n_1>n_0 : \lim_{N\longrightarrow \infty}\frac{1}{N} \sum_{i=0}^{N-1} \phi_{n_1}(T_{n_1}^iy) \not\longrightarrow \int_{\mathbb{T}^2} \phi_{n_1} d\mu.$$
     %\ \ \text{if} \ \  l\leq 2^{k-2} \ \ \forall \  k\leq \label{eq:6.1a} n\\
    %T_n^m(\T^1\times J_l^k)&\subset \T^1 \times [{1}/{2},1) \ \ \text{if} \ \  l\geq 2^{k-2} \ \ \forall  \ k\leq n\label{eq:6.1b}
% With the convergence estimate (\ref{eq:6.24}), it is enough to show there exist a continuous function $\phi_n:\T^2\longrightarrow\mathbb{R}$ such that for $y\in \mathrm{NB}_n$, 
%$$\lim_{N\longrightarrow \infty}\frac{1}{N} \sum_{i=0}^{N-1} \phi_n(T_n^iy) \not\longrightarrow \int_{\mathbb{T}^2} \phi_n du.$$
%Every element from $\mathrm{NB}_n$ and by remark(\ref{rem:6b}), we reduced to 
It shows there are infinitely many $\{T_n\}$ whose orbit $\{T^i_n(y)\}_{i=0}^{q_n-1}$ is not uniformly distributed along the whole torus, and $y\in \mathrm{NB}$ is arbitrary. It completes the claim.
\end{proof}

%\subsection{Construction Of Generic sets}
\subsection{Proof of Theorem D}\label{sec:6.2a}
%%%%% general Cantor set %%%%%%%%
Here, we construct a couple of sets containing the generic points for the interesting values of their Hausdorff dimension. The sets can be constructed in a similar manner to the set $\mathrm{G}$ constructed in the last subsection (see \ref{eq:6a}). Therefore we will only mention the  remarkable changes that need to be made.\\
For any $1<\alpha<2,$ and consider a Cantor set $C_{\lambda}$ associated with the sequence $\lambda=\{\lambda_k\}_{k\in \mathbb{N}},$ where $\lambda_k= \frac{1}{c_0}(\frac{1}{k})^{\frac{1}{\alpha-1}},$ the constant $c_0= \sum_{k\in \mathbb{N}}\lambda_k,$ explained in section {\ref{sec:2.3a}}.  
At first, just replace the Cantor set $C$ with $C_{\lambda}$, $I_l^n$ with $I_{l,\lambda}^n$, and $J_l^n$ with $J_{l,\lambda}^n$ in $(\ref{eq:6a}),\ref{eq:6.1.b}, (\ref{eqn:6a})$ and  $(\ref{eq:6b})$ to get following collection of disjoint subsets of $\T^2:$ $\T^2 = (\mathrm{G}_{\lambda} \cup \mathrm{NG}_{\lambda})$ where
 \begin{align}
     \mathrm{G}_{\lambda}&= \bigcap_{n\geq 1} \mathrm{G}_{n,\lambda} = \T^1 \times C_{\lambda} ,\ \  \text{where} \ \mathrm{G}_{n,\lambda} = \T^1 \times \bigcup_{l=0}^{2^{n}-1} I_{l,\lambda}^n , \label{eq:6.2a} \ \ \\
     \mathrm{NG}_{\lambda} &= \bigcup_{n\geq 1} \mathrm{NG}_{n,\lambda} = \T^1 \times ([0,1]\backslash C_{\lambda}), \ \ \text{where}
     \ \ \mathrm{NG}_{n,\lambda} = \T^1 \times\bigcup_{k=0}^{n-1} \bigcup_{l=0}^{2^{k-1}-1} J_{l,\lambda}^k ,\label{eq:6.2b} \ \  
    \end{align}
where $I_{l,\lambda}^n$ and $J_{l,\lambda}^n$ are intervals of $[0,1]$ as defined in section {\ref{sec:2.3a}}. We split the interval $J_{0,\lambda}^1$ into two equal halves as $J_{0,\lambda}^1= \hat{J}_{0,\lambda}^1\cup \hat{J}_{1,\lambda}^1$.\\
Consider the following permutation map $\widetilde{\phi}_{n,\lambda}: \T^2 \longrightarrow \T^2$ which follows the same combinatorics as $\widetilde{\phi}_n$ from section $\ref{sec:6a}$.
%$\widetilde{\phi}_{n,\lambda}: \left[\frac{i}{q_n},\frac{i+1}{q_n}\right) \times \T^1 \longrightarrow \left[\frac{i}{q_n},\frac{i+1}{q_n}\right) \times \T^1$ as 
\begin{align}
    \widetilde{\phi}_{n,\lambda}\left(\left[ \frac{i}{q_n},\frac{i+1}{q_n}\right) \times I_{l,\lambda}^{n} \right) &= \left[ \frac{i}{q_n} + \sum_{k=0}^{l-1}\frac{|I_{k,\lambda}^n|}{q_n},\frac{i}{q_n}+ \sum_{k=0}^{l}\frac{|I_{k,\lambda}^{n}|}{q_n}\right) \times \T^1  \ \ \forall \ 0\leq l< 2^n \\
   %\widetilde{\phi}_{n,\lambda}\left(\left[ \frac{i}{q_n},\frac{i+1}{q_n}\right) \times J_{l,\lambda}^{n} \right) &= \left[ \frac{i}{q_n} + \sum_{l=0}^{2^n-1}\frac{|I_{l,\lambda}^{n}|}{q_n},\frac{i}{q_n}+ \sum_{l=0}^{2^n-1}\frac{|I_{l,\lambda}^{n}|}{q_n} + \sum_{l=0}^{2^{n-1}-1}\frac{|I_l^{n-1, \lambda}|}{q_n}\right) \times \left[\frac{l}{2^{n-1}} , \frac{l+1}{2^{n-1}}\right) \\
   %&\forall 0\leq l< 2^{n-1} \nonumber \ \ l=0,1,..,2^{n-1}-1 \nonumber
       \widetilde{\phi}_{n,\lambda}\left(\left[ \frac{i}{q_n},\frac{i+1}{q_n}\right) \times J_{l,\lambda}^{k} \right) &= \left[ \frac{i}{q_n} + \sum_{l=0}^{2^k-1}\frac{|I_{l,\lambda}^{k}|}{q_n},\frac{i}{q_n}+ \sum_{l=0}^{2^k-1}\frac{|I_{l,\lambda}^{k}|}{q_n} +
       \sum_{l=0}^{2^{k-1}-1}\frac{ 2^{n-1}|J_{l,\lambda}^k|}{q_n}\right)
       %\sum_{l=2^{n-1}}^{2^{n}-1}\frac{ \lambda_l}{q_n}\right)
       \times \left(\frac{l}{2^{n-1}} , \frac{l+1}{2^{n-1}}\right); \nonumber \\
       &\qquad \qquad \qquad \qquad  \qquad \qquad \qquad \qquad \forall \  0\leq l< 2^{k-1}, \ \  2\leq k\leq n, \\
        \widetilde{\phi}_{n,\lambda}\left(\left[ \frac{i}{q_n},\frac{i+1}{q_n}\right) \times \hat{J}_{l,\lambda}^{1} \right) &= \left[ \frac{i}{q_n} + \sum_{l=0}^{1}\frac{|I_{l,\lambda}^{1}|}{q_n},\frac{i}{q_n}+ \sum_{l=0}^{1}\frac{|I_{l,\lambda}^{1}|}{q_n} +
       \frac{ 2|\hat{J}_{l,\lambda}^1|}{q_n}\right)
       %\sum_{l=2^{n-1}}^{2^{n}-1}\frac{ \lambda_l}{q_n}\right)
       \times \left(\frac{l}{2} , \frac{l+1}{2}\right) \ \ \forall \  l=0,1 \nonumber
   % \ \ l=0,1,..,2^{n-1}-1 \nonumber
\end{align}
Then the final conjugation map $h_n:\T^2\longrightarrow\T^2$ can be described as \begin{align}
    h_n= \overline{\phi}_{n,\lambda}\circ P_n
\end{align}
where $\overline{\phi}_{n,\lambda}$ is a smooth approximation of the map $\widetilde{\phi}_{n,\lambda}$ and diffeomorphism $P_n$ with the same smooth map $\kappa_n:\T^1\longrightarrow[0,1]$ from  $(\ref{eqn:1.2})$ with $\delta_n=\lambda_{2^{n+1}}$.
To exclude the region where we don't have control over the combinatorics, we consider a subset $E_n$ of $\T^1$ as 
\begin{align}
E_n= \left(\bigcup_{i=0}^{s_nq_n-1}\left[\frac{i}{s_nq_n}-\frac{\epsilon_n'}{2},\frac{i}{s_nq_n}+\frac{\epsilon_n'}{2}\right] \times \T^1\right)\bigcup \left( \bigcup_{l=0}^{2^n-1} \T^1 \times \left[I_{l,\lambda}^n-\frac{\epsilon_n'}{2},I_{l,\lambda}^n +\frac{\epsilon_n'}{2}\right]\right) %\bigcup_{l=0}^{2^{n-1}-1}\left[J_{l,\lambda}^n-\epsilon_n',J_{l,\lambda}^n +\epsilon_n'\right]\\ 
%E_n^2= 
\end{align}
where $\epsilon_n'$ is chosen such that $\mu(E_n)< \frac{1}{e^{3^n}}.$ Denote the set $F_n= \T^2\backslash \E_n$ such that $\mu(F_n)> 1-\frac{1}{e^{3n}}.$\\
Analogously, we consider the specific domains as in (\ref{eq:6f}). Using  $\frac{\delta_n}{n^2}\leq |I_{l,\lambda}^n|,$ for all $l=0,1,\ldots,2^{n}-1,$ we produce the following result as similar to lemma \ref{le:6a} and proposition \ref{pr:2a} as
\begin{prop}\label{pr:6.2a}
For $\epsilon>0$, consider  $(\frac{\sqrt{2}}{q_{n}}, \epsilon)$-uniformly continuous function $\psi : \mathbb{T}^2 \longrightarrow \mathbb{R}$, i.e. $\psi(B_\frac{\sqrt{2}}{q_{n}}{(x)})\subset B_{\varepsilon}(\psi(x))$. Then for any
 $x\in \mathrm{G}_{n,\lambda}$ satisfy the following estimate:
\begin{align}
\left|\frac{1}{q_{n+1}} \sum_{i=0}^{q_{n+1}-1} \psi(h_n \circ S_{\alpha_{n+1}}^{i}(x)) -   \int_{\T^2} \psi d\mu \right| \leq 4\epsilon + \frac{2}{n^4}\|\psi\|_0 
\end{align}
\end{prop}

The proof of theorem D will follow on the same line as the proof of theorem C. We start by choosing $L_n$ to be uniform Lipshitz constant and $q_{n+1} = l_nq_n^2$ where $l_n$ satisfying (\ref{eq:6.1c}). Now it is enough to show that
every point inside $\mathrm{B}_{\lambda}= \liminf_{n\rightarrow\infty}\mathrm{B}_{n,\lambda}$ where $\mathrm{B}_{n,\lambda}= H_n(\mathrm{G}_{\lambda})$ is a generic point, and its Hausdorff dimension lies between $\alpha-1$ and $\alpha$. The latter fact is followed by using proposition (\ref{pr:6.2a}) as done in claim 2, and  $\text{dim}_H(C_{\lambda})= \alpha-1 $ and
$\text{dim}_H(\mathrm{G}_{\lambda})= \alpha$ followed by (\ref{eq:6.1d}) and (\ref{def:1b}). \\
In our specific case, the same relations as mentioned in remark {\ref{rem:6b}} are satisfied, and hence, it shows that 
every point inside the $\mathrm{NB}_{\lambda}=\mathbb{T}^2\backslash{\mathrm{B}_{\lambda}}$ is a non-generic point. This completes the proof.
\subsection{Proof of Theorem E:}
To prove the theorem, we divide  $\T^2$ into two disjoint subsets where one subset supports an ergodic measure, and the other subset has measure zero, and its  Hausdorff dimension  is less than $\alpha$, which contains all non-generic points. For that, we follow a similar construction for the map $T\in \text{Diff}^{\infty}(\T^2,\mu)$ as done in the proof of theorem D. Hereby, we present the modification in the combinatorics of the elements of $\T^2= \mathrm{G}_{\lambda}\cup\mathrm{NG}_{\lambda,}$ which allows us to prove set $\mathrm{G}_{\lambda}$ by (\ref{eq:6.2a}) and set $\mathrm{NG}_{\lambda}$ by (\ref{eq:6.2b}) traps only non-generic points and generic points, respectively.\\
Consider the following permutation map $\widetilde{\phi}_{n,\lambda}: \T^2 \longrightarrow \T^2$, in place of $\widetilde{\phi}_{n,\lambda}$ from section $\ref{sec:6.2a}$, which follows the required combinatorics as (for $i= 0,1,\ldots,q_n-1$, \ $l= 0,1,\ldots,2^{n}-1$ and $k\leq n$), 
%$\widetilde{\phi}_{n,\lambda}: \left[\frac{i}{q_n},\frac{i+1}{q_n}\right) \times \T^1 \longrightarrow \left[\frac{i}{q_n},\frac{i+1}{q_n}\right) \times \T^1$ as 
\begin{align}
    \widetilde{\phi}_{n,\lambda}\left(\left[ \frac{i}{q_n},\frac{i+1}{q_n}\right) \times I_{l,\lambda}^{n} \right) &= \left[ \frac{i}{q_n} + \sum_{k=1}^{n}\sum_{j=0}^{2^{k-1}-1}\frac{ |J_{j,\lambda}^k|}{q_n}, \ \frac{i}{q_n}+ \sum_{k=1}^{n}\sum_{j=0}^{2^{k-1}-1}\frac{ |J_{j,\lambda}^k|}{q_n}+ \sum_{k=0}^{2^{n}-1}\frac{2^{n}|I_{k,\lambda}^{n}|}{q_n}\right) \times \left(\frac{l}{2^{n}}, \frac{l+1}{2^n}\right) \nonumber\\
    %\ \ \forall \ 0\leq l< 2^n \\
   %\widetilde{\phi}_{n,\lambda}\left(\left[ \frac{i}{q_n},\frac{i+1}{q_n}\right) \times J_{l,\lambda}^{n} \right) &= \left[ \frac{i}{q_n} + \sum_{l=0}^{2^n-1}\frac{|I_{l,\lambda}^{n}|}{q_n},\frac{i}{q_n}+ \sum_{l=0}^{2^n-1}\frac{|I_{l,\lambda}^{n}|}{q_n} + \sum_{l=0}^{2^{n-1}-1}\frac{|I_l^{n-1, \lambda}|}{q_n}\right) \times \left[\frac{l}{2^{n-1}} , \frac{l+1}{2^{n-1}}\right) \\
   %&\forall 0\leq l< 2^{n-1} \nonumber \ \ l=0,1,..,2^{n-1}-1 \nonumber
       \widetilde{\phi}_{n,\lambda}\left(\left[ \frac{i}{q_n},\frac{i+1}{q_n}\right) \times J_{l,\lambda}^{k} \right) &= \left[ \frac{i}{q_n} +
        \sum_{k'=1}^{k-1}\sum_{j=0}^{2^{k'-1}-1}\frac{ |J_{j,\lambda}^{k'}|}{q_n}+\sum_{j=0}^{l-1}\frac{ |J_{j,\lambda}^{k}|}{q_n}, \ \frac{i}{q_n} +
       \sum_{k'=1}^{k-1}\sum_{j=0}^{2^{k'-1}-1}\frac{ |J_{j,\lambda}^{k'}|}{q_n}+\sum_{j=0}^{l}\frac{ |J_{j,\lambda}^k|}{q_n}\right) \times \T^1 \nonumber
       %\ \ \forall 0\leq l< 2^{n-1} \nonumber
   % \ \ l=0,1,..,2^{n-1}-1 \nonumber
\end{align}

\remark Recall that $|J_{l,\lambda}^k|= \lambda_{2^{k-1}+l-1}$ for all $k\leq n$ and 
$|I_{l,\lambda}^{n}|= \sum_{n=k}^{\infty}\sum_{j=l2^{n-k}}^{(l+1)2^{n-k}-1}\lambda_{2^{n}+j}.$ Refer to Figure (\ref{fig:03}) for an illustration of the combinatorics.
%\sum_{j=0}^{l-1}\lambda_{2^{n-1}+j-1}+ \sum_{k=1}^{n-1}\sum_{j=0}^{2^{k-1}-1}\frac{\lambda_{2^{k-1}+j-1}}{q_n}$
\begin{figure}[ht]
    \centering
    \includegraphics[width=0.75\textwidth]{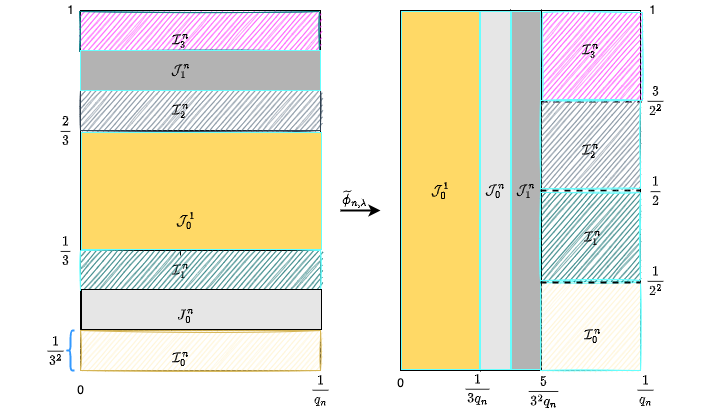}
    \caption{An example of action $\widetilde{\phi}_{n,\lambda}$ on the elements of $\mathrm{G}_n\cup\mathrm{NG}_n$ for $n=2.$}
    \label{fig:03}
\end{figure}
Following the analogous construction from section {\ref{sec:6.2a}}, we reduce to the following proposition for the elements of $\mathrm{NG}_{\lambda}$ and $\mathrm{G}_{\lambda},$ which is sufficient to prove the required property.
\begin{prop}\label{pr:6.3a}
\begin{enumerate}
    \item For $\epsilon>0$, consider  $(\frac{\sqrt{2}}{q_{n}}, \epsilon)$-uniformly continuous function $\psi : \mathbb{T}^2 \longrightarrow \mathbb{R}$, i.e. $\psi(B_\frac{\sqrt{2}}{q_{n}}{(x)})\subset B_{\epsilon}(\psi(x))$. Then for any
 $x\in \mathrm{NG}_{n,\lambda}$ satisfy the following estimate:
\begin{align}
\left|\frac{1}{q_{n+1}} \sum_{i=0}^{q_{n+1}-1} \psi(h_n \circ S_{\alpha_{n+1}}^{i}(x)) -   \int_{\T^2} \psi d\mu \right| \leq 4\epsilon + \frac{1}{2^{n(\alpha-1)}}\|\psi\|_0 
\end{align}
\item Every element $\T^1\times I_{l,\lambda}^n \in \mathrm{G}_{n,\lambda}$ satisfies
 \begin{align}
     h_n(\T^1\times I_{l,\lambda}^n)&\subset \T^1 \times \left[\frac{l}{2^n},\frac{l+1}{2^n}\right) 
 \end{align}
\end{enumerate}
\end{prop}
\remark Here, the set $\mathrm{B}_{\lambda}= \liminf_{n\rightarrow\infty}H_n({\mathrm{G}_{\lambda}})$ contains the non-generic points of the map $T$ and its $\alpha-1 \leq \dim_H(\mathrm{B}_{\lambda})\leq \alpha$ (see theorem D) for chosen
$\lambda=\{\lambda_k\}_{k\in \N}$ defined by $\lambda_k= \frac{1}{c_0}(\frac{1}{k})^{\frac{1}{\alpha-1}},$ the constant $c_0= \sum_{k\in \mathbb{N}}\lambda_k$.

\subsection{Future Direction:}
\begin{enumerate}
%    \item Is there exist a smooth ergodic diffeomorphism $T\in\text{Diff }^\infty(\T^2,\mu)$ constructed by the approximation of conjugation method, such that the set A containing all the generic point has $$\text{dim}_{H}(A)< 2.$$
    \item Can we choose a set $\mathrm{B}$ containing all the generic points such that $\text{dim}_H(\mathrm{B})= \alpha$ for all $0<\alpha<2$?
    \item Can we choose a generic set $\mathrm{B}$ of type $C\times C$, where $C$ is Cantor set on the unit interval, in the  above setup of theorem C? 
    \item Can we generalize the theorem C for a 3-dimensional torus with a choice of generic set of type
    \begin{itemize}
        \item $\mathrm{B}= \mathbb{T}^1\times C\times C.$ If this is true, the result generalizes to the n-dimensional torus.
        \item  In fact, can we choose the set $A= \mathbb{T}^1\times$``2D-fractal", where 2D fractal is not  necessarily the product of two sets like $C \times C $ type.  \end{itemize}
\end{enumerate} 
\begin{center}
\large{Acknowledgement}
\end{center}
The author wants to thank P.~Kunde and S.~Banerjee for suggesting the problem and for valuable discussions which helped to develop the ideas put forward.
The work was supported by the University Grants Commission(UGC-JRF), India.
\bibliographystyle{plain}
\bibliography{Reference.bib}
\end{document}